\documentclass[12pt]{amsart}
\usepackage[top=1in, bottom=1in, left=1.1in, right=1.1in]{geometry}

\usepackage{verbatim,dsfont}
\usepackage{amsmath}
\usepackage{amsfonts}
\usepackage{amssymb}
\usepackage{color}
\usepackage[all]{xy}
\usepackage{enumerate}
\usepackage{enumitem}
\usepackage{mathtools}
\usepackage{tensor}
\usepackage{tikz-cd}
\usepackage{hyperref}
\usepackage{cleveref}

\newtheorem*{theorem*}{Theorem}
\newtheorem{theorem}{Theorem}[section]
\newtheorem{lemma}[theorem]{Lemma}
\newtheorem{proposition}[theorem]{Proposition}

\newtheorem{conjecture}[theorem]{Conjecture}

\newtheorem{headthm}{Theorem}

\theoremstyle{definition}
\newtheorem{definition}[theorem]{Definition}
\newtheorem{notations}[theorem]{Notations}
\newtheorem{example}[theorem]{Example}
\newtheorem{settings}[theorem]{Settings}

\theoremstyle{remark}
\newtheorem{remark}[theorem]{Remark}

\newcommand{\str}{\textup{str}}
\newcommand{\ind}{\textup{ind}}
\newcommand{\pd}{\textup{pd}}
\newcommand{\het}{\textup{ht}}
\newcommand{\reg}{\textup{reg}}

\setlist[enumerate,1]{label=(\roman*)}

\begin{document}

\title{Explicit Stillman bounds for all degrees}

\author{Giulio Caviglia}
\address[G.~Caviglia]{Department of Mathematics, Purdue University, 150 N. University Street, West Lafayette, IN 47907-2067, USA.}
\email{\href{mailto:gcavigli@math.purdue.edu}{gcavigli@math.purdue.edu}}

\author{Yihui Liang}
\address[Y.~Liang]{Department of Mathematics, Purdue University, 150 N. University Street, West Lafayette, IN 47907-2067, USA.}
\email{\href{liangyihui1996@hotmail.com}{liangyihui1996@hotmail.com}}

\author{Cheng Meng}
\address[C.~Meng]{Yau Mathematical Sciences Center, Tsinghua University, Beijing 100084, China.}
\email{\href{mailto:cheng319000@mail.tsinghua.edu.cn}{cheng319000@mail.tsinghua.edu.cn}}

\thanks{Giulio Caviglia was partially supported by a grant from the Simons Foundation (41000748, G.C.)}
\maketitle

\begin{abstract}
In 2016 Ananyan and Hochster proved Stillman's conjecture by showing the existence of a uniform upper bound on the length of an $R_\eta$-sequence containing fixed $n$ forms of degree at most $d$ in polynomial rings over a field. This result yields many other uniform bounds including bounds on the projective dimension of the ideals generated by $n$ forms of degree at most $d$. Explicit values of these bounds for forms of degree $5$ and higher are not yet known.

This article constructs such explicit bounds, one of which is an upper bound for the projective dimension of all homogeneous ideals, in polynomial rings over a field, generated by $n$ forms of degree at most $d$. In the settings of the Eisenbud-Goto conjecture, we derive an explicit bound of the Castelnuovo-Mumford regularity of a nondegenerate prime ideal $P$ in a polynomial ring $S$ in terms of the multiplicity of $S/P$.
\end{abstract}

\section{Introduction}

Let $K$ be a field, $R=K[x_1,\dots,x_N]$ be a polynomial ring over $K$, and $I$ be an ideal of $R$ generated by $n$ forms of degrees $d_1,\dots,d_n$. We will denote the projective dimension of a module $M$ over $R$ by $\pd_R(M)$ or $\pd(M)$ when $R$ is clear from context. Stillman (see \cite{Stillman})  conjectured that $\pd_R(I)$ can be bounded in terms of $n$ and $d_1,\dots,d_n$ but independent of $N$. We will refer to such bounds as Stillman bounds. Ananyan and Hochster were the first to give an affirmative answer to Stillman's conjecture in \cite{AH1}, where they showed the existence of Stillman bounds by proving the existence of small subalgebras and small subalgebra bounds $\prescript{\eta}{}{B}$ (defined in \cite[Theorem B]{AH1}). Stillman's conjecture was later reproved in \cite{DLL} and \cite{ESS2}, both using topological Noetherianity results from \cite{Draisma}.

With the existence proven, the next question is to find explicit Stillman bounds. While many early and recent works \cite{AH2} \cite{BEDER20111105} \cite{Bruns_1976} \cite{burch_1968} \cite{10.2307/20535117} \cite{ENGHETA201060} \cite{ENGHETA2007715} \cite{Kohn_1972} \cite{MANTERO20191383} \cite{JaMc} have established Stillman bounds in degree at most $4$, the question for degree $5$ and higher remains untouched. In this paper, we will give effective answers to this question and other related questions asking bounds on certain data of ideals or modules.

In Ananyan and Hochster's paper \cite{AH1}, they gave an inductive process for the existence of Stillman bounds. Roughly speaking, consider $k$ sequences of functions $\{f_{ij}\}_{1 \leq i \leq k, j \geq 1}$, which is a table of functions of $k$ rows and infinitely many columns indexed by $j$, and functionals $\{\mathcal{F}_{ij}\}_{1 \leq i \leq k, j \geq 2}$ in suitably many functional variables such that it makes sense to talk about whether the following equality holds for any $i,j$:
$$f_{i+1,j}=\mathcal{F}_{ij}(f_{11},\ldots,f_{k,j-1},f_{1j},\ldots,f_{ij})$$
where we define $f_{k+1,j}=f_{1,j+1}$. When these equalities hold, we will say all the information on $\mathcal{F}_{ij}$ determines a recurrence relation. The main argument by Ananyan and Hochster says that \emph{there exist recurrence relations $\mathcal{F}_{ij}$ such that given initial conditions, that is, the concrete expression of $f_{ij}$ when $j$ is small, then there exist $k$ sequences $f_{ij}$, satisfying the recurrence relation, such that certain property $\mathcal{P}_i$ is satisfied by $f_{ij}$ for $1 \leq i \leq k$ and any $j$}. 

By choosing a suitable initial condition, one of the $\mathcal{P}_i$ turns out to be bounding the projective dimension of ideals generated in bounded degree by bounded number of elements. Thus, for this fixed $i$ and any $j$, the expression of $f_{ij}$, which is the image of finitely many $f_{kl}$'s under an iteration of $\mathcal{F}_{kl}$'s, gives us a bound on $\pd_R(I)$ only depending on the number of generators of $I$ and the maximal degree of generators. 

Two difficulties lie in this process. The first difficulty is about simplifying the iteration: instead of writing the final result by an inductive iteration of functions with many parentheses, is there an explicit function that only depends on few variables, like one variable or two variables? The second one is lack of concrete expression for a certain functional in this process. This functional takes an $f$ as an input and gives an output $\mathcal{F}(f)$ using iteration of functions. However, the number of iterations needed depends on the value of 
the function $f$ itself, making it impossible to express this functional in terms of arithmetic operations or exponential functions. 

In this paper, we will give an explicit Stillman bound. We overcome the first difficulty by applying Knuth arrows, and overcome the second one by introducing $F$-decompositions. Our results show that the Stillman bound on projective dimension can be given by two consecutive Knuth arrows, one of which has a quadratic degree. 

We follow the notation as in \cite{AH1}. In particular, for $\eta \in \mathbb{N}$, $R_\eta$ refers to Serre's condition, that is, a ring satisfies $R_\eta$ if it is regular in codimension $\eta$. The \Cref{Theorem A} below gives an explicit bound to a corollary of \cite[Theorem B]{AH1}.
\begin{headthm}(See Theorem \ref{5.9})\label{Theorem A}
Let $d \in \mathbb{Z}_{\geq 2}$, $\tilde{d}=d+1$, $\delta_1,\ldots,\delta_d \in \mathbb{N}$. Assume $\sigma=\max\{2,\delta_i:1\leq i \leq d\}$. Then for any graded vector space $V=\oplus_{i \geq 1}V_i$ with $\dim V_i=\delta_i$, if $I$ is the $R$-ideal generated by $V$, then $\pd(R/I) \leq \tilde{d}\uparrow^{d^2+d-2}\tilde{d}\uparrow\sigma$.     
\end{headthm}

The statements (i)-(vi) of \Cref{Theorem B} below realize the explicit bounds in Theorem D (c), Theorem A (a), Theorem A (b), Theorem B, Theorem C, and Theorem F in \cite{AH1} under the same setting respectively.
\begin{headthm}\label{Theorem B}
We have the following bounds for the properties below:
\begin{enumerate}
\item (See Theorem \ref{6.1}) Assume $d \geq 2$ and $D_d(\sigma)=d\uparrow2\uparrow \tilde{d}\uparrow^{d^2+d-2}\tilde{d}\uparrow\sigma$. Then for any $R$-ideal $I$ generated by a regular sequence of length at most $\sigma$ and degree at most $d$, if $P$ is a minimal prime over $I$, then $P$ is generated by at most $D_d(\sigma)$ elements. 
\item (See Theorem \ref{6.2}) Assume $\eta \geq 2$. We set $A_2(\eta)=\eta+2$, and for $d \geq 3$, set
$$A_d(\eta)=d\uparrow^{d^2-d-1}3\uparrow\eta.$$
Then for any $F \in R_d$ whose strength is at least $A_d(\eta)$, the ring $R/FR$ satisfies Serre's condition $R_\eta$.
\item (See Theorem \ref{6.3}) We set $\bar{A}_2(\eta,\delta)=\eta+4|\delta|+1$, and for $d \geq 3$, set $\bar{A}_d(\eta,\delta)=|\delta|-1+d\uparrow^{d^2-d-1}3\uparrow(\eta+3|\delta|)$. Then for any graded vector space $V=\oplus_{i \geq 1}V_i$ with $\dim V_i=\delta_i$, when the strength of 
$V_j$ is at least $\bar{A}_d(\eta,\delta)$ for $1 \leq j \leq d$, then any homogeneous $K$-basis of $V$ is an $R_\eta$-sequence. 
\item (See Theorem \ref{6.4}) Assume $d \geq 2, \eta \geq 4$. We set $d_1=\max\{d+1,\eta-1\}$, $B_d(\eta,\delta)=d_1\uparrow^{d^2+d-1}d_1\uparrow|\delta|$, then for any graded vector space $V=\oplus_{i \geq 1}V_i$ with $\dim V_i=\delta_i$, $V$ lies in a $K$-algebra generated by an $R_\eta$-sequence of degree at most $d$ and length at most $B_d(\eta,\delta)$.
\item (See Theorem \ref{6.5}) Assume $d \geq 2$ and set $\tilde{d}=d+1$. Let $M$ be an $R$-module with minimal presentation
$$R^n \xrightarrow[]{\phi} R^m \to M \to 0$$
where entries of $\phi$ are elements in $R$, which are not necessarily homogeneous and have degree at most $d$. Then $\pd_R(M) \leq \tilde{d}\uparrow^{d^2+d-2}\tilde{d}\uparrow(mn)$.
\item (See Theorem \ref{6.6}) For $d \geq 2$, there exists a function $\Phi_d(h)$ such that for $F \in R_d$ with strength at least $\Phi_d(h)$, the ideal generated by $\{\frac{\partial F}{\partial x_i}\}_{1 \leq i \leq N}$ is not contained in an ideal generated by $h$ forms of degree at most $d-1$. We can take $\Phi_d(h)=h+1$ for $d=2$ and $\Phi_d(h)=d\uparrow^{d^2-d-2}d\uparrow\max\{2,h\}+1$ for $d \geq 3$. 
\end{enumerate}
\end{headthm}
The bounds above are also related to the famous Eisenbud-Goto conjecture as follows:
\begin{conjecture}[\cite{EISENBUD198489}]\label{Conjecture 1.1}
Let $K$ be an algebraically closed field, $P \subset (x_1,\ldots,x_n)^2$ be a homogeneous ideal in $R=K[x_1,\ldots,x_n]$, then
$$\reg(P) \leq \deg(R/P)-\het(P)+1.$$
\end{conjecture}
The conjecture is proved for curves by Gruson-Lazarsfeld-Peskine, for smooth surfaces by Lazarsfeld and Pinkham, for most smooth 3-folds by Ran and Kwak, if $R/P$ is Cohen-Macaulay by Eisenbud-Goto, and in many other special cases.  In 2018, McCullough and Peeva found counterexamples to the Eisenbud-Goto Conjecture in \cite{MR3758150EGcounterexample}, and they further showed that $\reg(R/P)$ cannot be bounded by any polynomial in terms of $e(R/P)$. Actually by \cite[Theorem 5]{mccullough2024prime}, any such bound must be at least double exponential in terms of $e(R/P)$. Later, Caviglia et al. made use of the technique in \cite{AH1} and prove the existence of a bound of $\reg(R/P)$ that only depends on the multiplicity $e(R/P)$ in \cite{MR3936076EGimplicitbound}. However, a bound as an explicit function in terms of $e(R/P)$ is not known yet. With the same notations and assumptions as \Cref{Conjecture 1.1}, we derive an explicit bound in terms of Knuth arrows.
\begin{headthm}(See Theorem \ref{7.5})
Let $P$ be a prime in $R$ of height $h \geq 2$ with $\deg(R/P)=e \geq 2$. Then
\begin{align*}
\reg(R/P) \leq (eh)\uparrow^{e^2h^2-1}(h+3).
\end{align*}
If moreover $K$ is algebraically closed, then
\begin{align*}
\reg(R/P) \leq (e^2)\uparrow^{e^4-1}(e+2).
\end{align*}
\end{headthm}
The organization of the paper is as follows. Section 1 is the introduction to the whole paper. Section 2 introduces the basic notations on which Ananyan and Hochster's argument runs. Section 3 introduces packages of functions and one recurrence relation which gives us a package satisfying certain conditions, and one function in this package bounds $\pd_R(I)$. In Section 4, we give upper bounding packages for the package introduced in Section 3. We also introduce the $F$-decomposition which bounds the implicit functional in the package. In Section 5, we recall the notion of Knuth arrow and derive a bound on the projective dimension. In Section 6, we derive other types of bounds using the bound found in Section 5. Section 7 gives an explicit Eisenbud-Goto bound by using the bound in Section 5.

\section{Notation}
In this section, we recall some notations which are needed for stating the theorems and some lemmas used in these theorems.
\begin{notations}
We make the following assumptions and give the definitions.
\begin{enumerate}
\item Let $\mathbb{N}$ be the set of nonnegative integers, $\mathbb{Z}_+$ be the set of positive integers.
\item Throughout the paper, we assume $K$ is an algebraically closed field unless otherwise stated, $R=K[x_1,\ldots,x_N]$ is a polynomial ring over $K$ where $N$ is sufficiently large. $R$ is naturally graded by setting the degree of $x_i,1 \leq i \leq N$ to be 1 and the degree of elements of $K$ to be $0$. Let 
$R_d$ be elements of degree $d$ in $R$, $R^h$ be homogeneous elements of positive degree in $R$.
\item The space of dimension vectors is $E=\oplus_{i \geq 1}\mathbb{N}e_i$. It is a monoid in a free abelian group with basis indexed by $\mathbb{N}$. For each $d \geq 2$, set $E_d=\oplus_{1 \leq i \leq d}\mathbb{N}e_i \subset E$. If there is no confusion, we identify $\mathbb{N}^d \cong E_d$ via $(\delta_1,\ldots,\delta_n) \to \delta_1e_1+\ldots+\delta_ne_n$. For any element $\delta \in E$, let $\delta_i$ be its $i$-th coordinate. We have $\delta_i=0$ for $i \gg 0$. Denote $\deg\delta=\max\{i:\delta_i \neq 0\}$. We have $E_d=\{\delta:\deg\delta \leq d\}$, $E_d \subset E_{d+1}$, and $E=\cup_{d \geq 1}E_d$. For $\delta \subset E$, denote $|\delta|=\sum_{i \geq 1}\delta_i$. We denote $\tilde{e}_i=e_i+e_{i-1}+\ldots+e_1=(1,1,\ldots,1) \in E_i$. For two dimension vectors $\delta,\delta' \in E$, write $\delta \geq \delta'$ if $\delta_i \geq \delta'_i$ for all $i$. This gives a partial order on $E$. For this order, $\leq,>,<$ make sense similarly.
\item The symbol $V$ itself always denotes a finite dimensional $K$-vector space generated by elements in $R^h$; thus $V$ is also graded. Write $V=\oplus_{i \geq 1}V_i$ where $V_i \subset R_i$; we say the element $\sum_{i \geq 1}(\dim_K V_i)e_i \in E$ is the dimension vector of $V$, denoted by $\underline{\dim}V$. We define $\deg V=\deg \underline{\dim}V=\max\{i:V_i \neq 0\}$.  In particular, we have $\dim V=|\underline{\dim}V|$.
\item A function $f$ from a subset of a product of $\mathbb{N}$ and $\mathbb{Z}_+$ to $ \mathbb{N}$ is ascending, if we increase one coordinate and fix all other coordinates, then the value of $f$ increases when it exists. In one variable this just means $f$ is increasing. A function $f=(f_1,\ldots,f_{d})$ from a subset of a product of $\mathbb{N}$ and $\mathbb{Z}_+$ to $\mathbb{N}^{d}$ is ascending, if all its coordinates are ascending.
\item For a subset $\sigma \subset R$, $V(\sigma) \subset \mathbb{A}^N_K$ is the set of points where all elements of $\sigma$ vanish. For $\eta \in \mathbb{Z}_+$, the condition $R_\eta$ refers to Serre's condition, that is, a Noetherian ring $S$ satisfies $R_\eta$ if $S_p$ is regular for any $p \in \textup{Spec}(S)$ with $\textup{ht}(p) \leq \eta$. If the singular locus of $S$ is closed and defined by an $S$-ideal $J$, this is equivalent to the condition that $\textup{ht}(J) \geq \eta+1$.
\item A sequence of elements generating a proper ideal of a Noetherian ring $S$ is a prime sequence (respectively, an $R_\eta$-sequence), if the quotient of $S$ by any initial segment is a domain (respectively, satisfies $R_\eta$). 
\item For $F \in R^h$, denote $\mathcal{D}F$ the $K$-vector space generated by its partial derivatives $\{\partial F/\partial x_i,1 \leq i \leq N\}$. In particular, if $F \in R_d$, then $\mathcal{D}F \subset R_{d-1}$.
\end{enumerate}
\end{notations}
\begin{remark}
If $S=R$ is a polynomial ring, $\eta \geq 1$, then an $R_\eta$-sequence lying in $R^h$ is a prime sequence.     
\end{remark}
\begin{definition}[Strength of elements in $R^h$]
\begin{enumerate}
\item Let $F \in R^h$, $k \in \mathbb{Z}_+$. We say $F$ has a $k$-collapse if $F=\sum_{1 \leq i \leq k}f_ig_i$ where $f_i,g_i \in R^h$. In other words, $F$ lies in an ideal generated by $k$ elements in lower degree. Denote the strength of $F$, $\str(F)=\max\{k:F \text{ does not have a } k\text{-collapse} \}$. If $F \in R_1$, denote $\str(F)=+\infty$.
\item Let $V$ be a graded vector space. Denote $\str(V)=\min\{\str(F):F \in V \cap R^h\}$. That is, $\str(V)$ is the minimal strength of nonzero homogeneous elements in $V$.
\item We say $F$ is $k$-strong if $\str(F)\geq k$. We say $V$ is $k$-strong if $\str(V)\geq k$. 
\end{enumerate}
Note that $F$ is irreducible if and only if $\str(F) \geq 1$.
\end{definition}
Next, we consider the following conditions on certain functions, which originate from \cite{AH1}.
\begin{definition}
\begin{enumerate}
\item Property $PAa$ refers to properties on a sequence of functions $A_d:\mathbb{Z}_+ \to \mathbb{Z}_+, d \in \mathbb{Z}_+$. We say the sequence $A_d$ satisfies Property $PAa$ at degree $d$, if $A_d$ is increasing, and we have: for any $\eta \in \mathbb{Z}_+$, $0 \neq F \in R^h$ with $\deg(F) \leq d$, $\str(F) \geq A_d(\eta)$ implies the codimension of the singular locus of $R/FR$ in $R/FR$ is at least $\eta+1$, that is, $R/FR$ is $R_\eta$.
\item Property $PAb$ refers to properties on a sequence of functions $\bar{A}_d:\mathbb{Z}_+ \times \mathbb{N}^d \to \mathbb{Z}_+, d \in \mathbb{Z}_+$. We say the sequence $\bar{A}_d$ satisfies Property $PAb$ at degree $d$, if $\bar{A}_d$ is ascending, and we have: for any $(\eta,\delta) \in \mathbb{Z}_+ \times \mathbb{N}^d$, $\underline{\dim}V=\delta$, if $\str(V_j) \geq \bar{A}_d(\eta,\delta)$ for all $1 \leq j \leq d$, then every sequence of $K$-linearly independent forms in $V$ is an $R_\eta$-sequence.
\item Property $PB$ refers to properties on a sequence of functions $B_d:\mathbb{Z}_+ \times \mathbb{N}^d \to \mathbb{Z}_+, d \in \mathbb{Z}_+$. We say the sequence $B_d$ satisfies Property $PB$ at degree $d$, if $B_d$ is ascending, and we have: for any $(\eta,\delta) \in \mathbb{Z}_+ \times \mathbb{N}^d$, $\underline{\dim}V=\delta$, then $V$, and hence the $K$-subalgebra generated by $V$ in $R$, is contained in a $K$-subalgebra generated by an $R_\eta$-sequence of forms of degree at most $d$ and length at most $B_d(\eta,\delta)$.
\item Property $PcB$ refers to properties on a sequence of functions $\mathcal{B}_d:\mathbb{Z}_+ \times \mathbb{Z}_+ \to \mathbb{Z}_+, d \in \mathbb{Z}_+$. We say the sequence $\mathcal{B}_d$ satisfies Property $PcB$ at degree $d$, if $\mathcal{B}_d$ is ascending, and we have: for any $(\eta,n) \in \mathbb{Z}_+ \times \mathbb{Z}_+$, if $\dim V=n$ and $\deg V \leq d$, then $V$, and hence the $K$-subalgebra generated by $V$ in $R$, is contained in a $K$-subalgebra generated by an $R_\eta$-sequence of forms of degree at most $d$ and length at most $\mathcal{B}_d(\eta,n)$.
\item Property $PD$ refers to properties on a sequence of functions $D_d:\mathbb{Z}_+ \to \mathbb{Z}_+, d \in \mathbb{Z}_+$. We say the sequence $D_d$ satisfies Property $PD$ at degree $d$, if $D_d$ is increasing, and if we assume $I$ is generated by a regular sequence of degrees at most $d$ of length at most $k$, then every minimal prime over $I$ is generated by at most $D_d(k)$ elements.
\end{enumerate}    
\end{definition}
In the definitions (ii) and (iii) above we use the identification $E_d=\mathbb{N}^d$. We abbreviate ``property $*$ at degree $d$" by $*_d$, so we get $PAa_d,PAb_d,PB_d,PcB_d,PD_d$ for $d \in \mathbb{Z}_+$. We will abbreviate ``$A_d$ satisfies Property $PAa_d$" by $A_d \in PAa_d$, and $\in$ means satisfying certain property instead of elements belonging to a set here. Note that by definition, if $d'\leq d$ and $A_d$ satisfies $PAa_d$, then it also satisfies $PAa_{d'}$; similarly for $PcB_d$ and $PD_d$. If $d'\leq d$ and $\bar{A}_d$ satisfies $PAb_d$, then when we restrict from $E_d$ to $E_{d'}$, the restriction of $\bar{A}_d$ satisfies $PAb_{d'}$, and similar for $PB_d$.
\begin{remark}
In the above definitions, the notations are slightly different from \cite{AH1} where we change the lower indices and the variables. For example, the function $\mathcal{B}_d(\eta,n)$ in this paper corresponds to $\prescript{\eta}{}{\mathcal{B}}(n,d)$ in \cite{AH1}. The action of pulling down the symbol $d$ to a lower index is essential because it marks our reduction steps; in each step we work from $d$ to $d+1$ for all the other variables.    
\end{remark}

Next we will give a description of an inductive process on sequences of functions. Such descriptions can be made clear if we introduce the following notations on packages of functions, recurrence relations and initial conditions.
\begin{definition}
\begin{enumerate}
\item We will say two functions are of the same type if their domain and range are the same. We will say two functionals have the same form if their input functions and output functions are of the same type.
\item Let $k \in \mathbb{Z}_+$. We say a 
\textbf{ $k$-package $\mathcal{P}$ of functions} is a set of functions $f_{i,j}$ indexed by $(i,j) \in \{1,\ldots,k\}\times \mathbb{Z}_+$. For such a package, we give an order $\leq$ on its index set $\{1,\ldots,k\}\times \mathbb{Z}_+$: $(i_1,j_1) \leq (i_2,j_2)$ if $j_1<j_2$ or $j_1=j_2,i_1<i_2$.
\item A recurrence relation $\mathcal{R}$ for a package is a set of functionals $\mathcal{F}_{i,j}$, $(i,j) \in \{1,\ldots,k\}\times \mathbb{Z}_+$, $(i,j) \neq (1,1)$ such that the input of $\mathcal{F}_{i,j}$ is a subset of $\{f_{i',j'}, (i',j')<(i,j)\}$ and the output has same type as $f_{i,j}$. We write the recurrence relation in the form $f_{i,j}=\mathcal{F}_{i,j}((f_{i',j'})_{(i',j')<(i,j)})$, and say two recurrence relations have the same form if the indices appearing in the variables and values of the functionals are the same. We say a package of functions $\mathcal{P}$ satisfies the recurrence relation $\mathcal{R}$ if this equality holds for any $(i,j)>(1,1)$.
\item An initial condition 
$\mathcal{I}$ of a package is the value of a finite subset of the package, which is a set of functions. We say two initial conditions $\mathcal{I}$, $\mathcal{I}'$ have the same form if the index sets of the functions are the same and the corresponding functions are of the same type.
\end{enumerate}
\end{definition}
A recurrence relation with an initial condition which is compatible and large enough determines a package. For example, a $2$-package will be the set of functions
$$f_{11}, f_{21}, f_{12}, f_{22}, f_{13}, f_{23},\ldots$$
and one recurrence relation for this package consists of functionals, which satisfies
$$f_{21}=\mathcal{F}_{21}(f_{11}),f_{12}=\mathcal{F}_{12}(f_{11},f_{21}),f_{22}=\mathcal{F}_{22}(f_{11},f_{21},f_{12}),\ldots.$$
This relation, along with the initial condition $f_{11}$, determines a package. We can draw a diagram on how the recurrence relation runs, where we use the symbol $\xRightarrow{}$ for variables and values of functionals:
\begin{center}
\begin{tikzcd}
f_{11} \arrow[r,Rightarrow]\arrow[d,Rightarrow]\arrow[rd,Rightarrow] & f_{21}\arrow[d,Rightarrow]\arrow[ld,Rightarrow]\\
f_{12} \arrow[r,Rightarrow] & f_{22}
\end{tikzcd}    
\end{center}
But if the recurrence relation of the package has the form
$$f_{21}=\mathcal{F}_{21}(f_{11}),f_{12}=\mathcal{F}_{12}(f_{21}),f_{22}=\mathcal{F}_{22}(f_{12}),\ldots,$$
then the recurrence relation runs like:
\begin{center}
\begin{tikzcd}
f_{11} \arrow[r,"\mathcal{F}_{21}",Rightarrow] & f_{21}\arrow[ld,"\mathcal{F}_{12}",Rightarrow]\\
f_{12} \arrow[r,"\mathcal{F}_{22}",Rightarrow] & f_{22}
\end{tikzcd}    
\end{center}

\begin{definition}
\begin{enumerate}
\item Let $f,g$ be two functions defined on the same set and maps to $\mathbb{R}$. We say $f \geq g$ if on every point, the value of $f$ is bigger than that of $g$.
\item We say a functional is ascending if we replace every input function by a bigger one, the value of the functional is a bigger value or a bigger function.
\item Let $\mathcal{F}$, $\mathcal{G}$ be two functionals, then we say $\mathcal{F} \geq \mathcal{G}$, if under the same input, the output function of $\mathcal{F}$ is bigger than or equal to that of $\mathcal{G}$.
\item We say a functional is positive if the input functions being ascending implies that the output function is also ascending.
\end{enumerate}
\end{definition}
\begin{definition}
\begin{enumerate}
\item If $\mathcal{R}_1$, $\mathcal{R}_2$ have the same form, we say $\mathcal{R}_1 \geq \mathcal{R}_2$, if all the functionals in $\mathcal{R}_1$ are bigger than their counterparts in $\mathcal{R}_2$ 
\item If $\mathcal{I}_1$, $\mathcal{I}_2$ have the same form, we say $\mathcal{I}_1 \geq \mathcal{I}_2$, if all the functions in $\mathcal{I}_1$ are bigger than their counterparts in $\mathcal{I}_2$ 
\item We say a package or an initial condition is ascending if it consists of ascending functions.
\item  We say a recurrence relation is ascending (respectively, positive) if it consists of ascending (respectively, positive) functionals.
\end{enumerate}    
\end{definition}
The following proposition is a summary of the inductive process on upper bounds. We will frequently apply this proposition to realize a new package as an upper bound for an old package.
\begin{proposition}
Let $\mathcal{P}$ be a $k$-package determined by recurrence relation $\mathcal{R}$ and initial condition $\mathcal{I}$, $\mathcal{P}'$ be a $k$-package determined by recurrence relation $\mathcal{R}'$ and initial condition $\mathcal{I}'$. Assume:
\begin{enumerate}
\item $\mathcal{R},\mathcal{I}$ and $\mathcal{R}',\mathcal{I}'$ have the same form;
\item $\mathcal{I} \geq \mathcal{I}'$, $\mathcal{R} \geq \mathcal{R}'$;
\item Functionals in $\mathcal{R}$ and $\mathcal{R}'$ are ascending but not necessarily positive.
\end{enumerate}
Then $\mathcal{P} \geq \mathcal{P'}$.
\end{proposition}
\begin{proof}
We prove $f_{\mathcal{P},ij} \geq f_{\mathcal{P}',ij}$ by induction on the index $(i,j)$.  
\end{proof}
Since positive functionals keep ascending property, we have:
\begin{proposition}
Let $\mathcal{P}$ be a $k$-package determined by recurrence relation $\mathcal{R}$ and initial condition $\mathcal{I}$. Assume $\mathcal{I}$ is ascending, $\mathcal{R}$ is positive but not necessarily ascending, then  $\mathcal{P}$ is ascending.  
\end{proposition}
In the following sections, it is usually trivial to see that the finite set $\mathcal{I}$ is ascending and the finite sequence of functionals $\mathcal{R}$ consists of ascending and positive functionals. By the above proposition we get that all the functions in the package are ascending. Thus we may work in the following settings in all the following sections with no packages violating it:
\begin{settings}\label{2.11}
We assume the packages and initial conditions are ascending and the recurrence relations are ascending and positive, unless otherwise stated.
\end{settings}
\begin{definition}
A special package is a $5$-package $\mathcal{P}=\{A_{\mathcal{P},d},\bar{A}_{\mathcal{P},d},B_{\mathcal{P},d},\mathcal{B}_{\mathcal{P},d},D_{\mathcal{P},d},d \in \mathbb{Z}_+\}$ such that the domain and the range of the functions are as follows: $A_{\mathcal{P},d}:\mathbb{Z}_+ \to \mathbb{Z}_+$, $\bar{A}_{\mathcal{P},d}:\mathbb{Z}_+ \times \mathbb{N}^d \to \mathbb{Z}_+$, $B_{\mathcal{P},d}:\mathbb{Z}_+ \times \mathbb{N}^d \to \mathbb{Z}_+$, $\mathcal{B}_{\mathcal{P},d}:\mathbb{Z}_+ \times \mathbb{Z}_+ \to \mathbb{Z}_+$, $D_{\mathcal{P},d}:\mathbb{Z}_+ \to \mathbb{Z}_+$. If the package is clear, we omit the $\mathcal{P}$'s in the lower indices.   
\end{definition}
\begin{definition}
Under Settings \ref{2.11}, we say a special package $\mathcal{P}$ is an Ananyan-Hochster package (AH package for short), if the set of functions $A_{\mathcal{P},d},\bar{A}_{\mathcal{P},d},B_{\mathcal{P},d},\mathcal{B}_{\mathcal{P},d},D_{\mathcal{P},d}$ satisfies $PAa_d, PAb_d,$
$ PB_d, PcB_d, PD_d$. We say a recurrence relation $\mathcal{R}$ is AH, if $\mathcal{R}$, together with any compatible initial condition satisfying $PAa_d, PAb_d, PB_d,$
$ PcB_d, PD_d$ for small $d$, determines an AH package. 
\end{definition}
\begin{remark}
We summarize what packages we will have in the following sections. In all, we have 5 packages:
\begin{enumerate}
\item $\mathcal{P}_\alpha$: 5-package with functions $A_d,\bar{A}_d,B_d,\mathcal{B}_d,D_d$;
\item $\mathcal{P}_\beta$: 2-package with functions $\bar{A}_d,B_d$;
\item $\mathcal{P}_\gamma$: 2-package with functions $\bar{A}_d,B_d$;
\item $\mathcal{P}_\zeta$: 2-package with functions $\hat{A}_d,\hat{B}_d$;
\item $\mathcal{P}_\theta$: 2-package with functions $A_d,B_d$.
\end{enumerate}
Roughly speaking, we will derive 2 sequences of inequalities
$$\bar{A}_{\mathcal{P}_\alpha,d}\leq \bar{A}_{\mathcal{P}_\beta,d}\leq \bar{A}_{\mathcal{P}_\gamma,d}\leq \hat{A}_{\mathcal{P}_\zeta,d}\leq A_{\mathcal{P}_\theta,d}$$
and
$$B_{\mathcal{P}_\alpha,d}\leq B_{\mathcal{P}_\beta,d}\leq B_{\mathcal{P}_\gamma,d}\leq \hat{B}_{\mathcal{P}_\zeta,d}\leq B_{\mathcal{P}_\theta,d}.$$
Then we use the fact that $\mathcal{P}_\alpha$ is an AH package to prove $\bar{A}_{\mathcal{P}_\alpha,d}$, $B_{\mathcal{P}_\alpha,d}$ satisfy the corresponding AH conditions, therefore an upper bound for $\bar{A}_{\mathcal{P}_\theta,d}$, $B_{\mathcal{P}_\theta,d}$ will also satisfy the corresponding AH conditions. However, the functions may have different domains and ranges. Some are defined on integers while others are defined on dimension vectors. The precise statement will be clear in the following sections.
\end{remark}

Next, we introduce notations that are needed for the proofs of Lemma \ref{boundcolon} and Lemma \ref{boundP}. 
\begin{notations}
\begin{enumerate}
\item For a finitely generated graded $R$-module $M$, let $\beta_{ij}(M)$ be the graded Betti numbers of $M$ and $\beta_i(M)=\sum_{j}\beta_{ij}(M)$ be the $i$-th Betti number of $M$. The Castelnuovo-Mumford regularity of $M$ is defined as $\reg(M)=\max_{i,j}\{j-i : \beta_{ij}(M)\neq 0\}$.
\item Let $J$ be a monomial ideal. We say $J$ is \textit{strongly stable} if for each monomial $u$ of $J$, $x_i|u$ implies $x_ju/x_i\in I$ for each $j<i$. Let $G(J)$ be the set of minimal monomial generators of $J$ and $D(J)$ be the largest degree of monomials in $G(J)$. If $u$ is a monomial, let $m(u):=\max\{i:x_i|u\}$. By the Eliahou-Kervaire resolution in \cite{EK}, if $J$ is strongly stable, then $\beta_i(J)=\sum_{u\in G(J)} \binom{m(u)-1}{i}$.
\item Let $I$ be a monomial ideal in $K[x_1,\dots,x_N]$ where $K$ is an infinite field, we may assume $I$ is generated by monic monomials, if $K^\prime$ is any other field, then let $I_{K^\prime}$ be the ideal generated by the image of these monomials in $K^\prime[x_1,\dots,x_N]$. Let $\text{gin}_{\text{rlex}}(I)$ be the generic initial ideal of $I$ with respect to the degree reverse lexicographical order. The \textit{zero-generic initial ideal} of $I$ with respect to the degree reverse lexicographical order is defined to be
\[
\text{Gin}_0(I):= (\text{gin}_{\text{rlex}}((\text{gin}_{\text{rlex}}(I))_\mathbb{Q}))_K.
\]
The zero-generic initial ideal is explored in more details in \cite{CS2}. We need this notion in $\S3$ to treat the positive characteristic cases. Notice that in characteristic $0$, the zero-generic initial ideal is equal to the usual generic initial ideal.
\end{enumerate}    
\end{notations}

\section{One AH recurrence relation $\mathcal{R}_{\alpha}$}
Using the notations introduced in Section 2, we can reinterpret Ananyan and Hochster's proof as proving certain recurrence relation $\mathcal{R}_{\alpha}$ is an AH recurrence relation. After that, we may start with any initial condition satisfying AH conditions, and such conditions are easy to find. The recurrence relation $\mathcal{R}_{\alpha}$ is on a special package. It has the form:
$$A_d=\mathcal{F}_{1,d}(D_{d-1},\mathcal{B}_{d-1}),\bar{A}_d=\mathcal{F}_{2,d}(A_d),B_d=\mathcal{F}_{3,d}(\bar{A}_d),\mathcal{B}_d=\mathcal{F}_{4,d}(B_d),D_d=\mathcal{F}_{5,d}(\mathcal{B}_d)$$
and the diagram of this relation looks like:
\begin{center}
\begin{tikzcd}[column sep=small,row sep=large]
& & &\mathcal{B}_{d-1}\in PcB_{d-1}\arrow[r,"\mathcal{F}_{5,d-1}",Rightarrow]\arrow[llld,"\mathcal{F}_{1,d}"',Rightarrow]&D_{d-1} \in PD_{d-1}\arrow[lllld,"\mathcal{F}_{1,d}",Rightarrow]\\
A_d \in PAa_d \arrow[r,"\mathcal{F}_{2,d}"',Rightarrow] & \bar{A}_d \in PAb_d\arrow[r,"\mathcal{F}_{3,d}",Rightarrow]& B_d \in PB_d\arrow[r,"\mathcal{F}_{4,d}",Rightarrow]&\mathcal{B}_d\in PcB_d\arrow[r,"\mathcal{F}_{5,d}",Rightarrow]\arrow[llld,"\mathcal{F}_{1,d+1}"',Rightarrow]&D_d \in PD_d\arrow[lllld,"\mathcal{F}_{1,d+1}",Rightarrow]\\
A_{d+1} \in PAa_{d+1} & & & &
\end{tikzcd}    
\end{center}
Among all the functionals in the relation $\mathcal{R}_\alpha$, only the functional $\mathcal{F}_{3,d}$ is not explicit. First we deal with $\mathcal{F}_{1,d}$, $\mathcal{F}_{2,d}$, $\mathcal{F}_{4,d}$. Then we apply a lemma by Chardin and some knowledge of generic initial ideal to derive $\mathcal{F}_{5,d}$. Finally we will introduce the concept of $F$-decomposition to get an inductively explicit formula for $\mathcal{F}_{3,d}$.
\begin{definition}
Let $d \geq 2$. Consider the following functions:
\begin{enumerate}
\item $A_d: \mathbb{Z}_+ \to \mathbb{Z}_+,$
\item $\bar{A}_d:\mathbb{Z}_+\times\mathbb{N}^d \to \mathbb{Z}_+,$
\item $B_d:\mathbb{Z}_+\times\mathbb{N}^d \to \mathbb{Z}_+,$
\item $\mathcal{B}_{d-1},\mathcal{B}_d:\mathbb{Z}_+\times\mathbb{Z}_+ \to \mathbb{Z}_+,$
\item $D_{d-1},D_d:\mathbb{Z}_+\to \mathbb{Z}_+.$
\end{enumerate}
We define the following functionals:
\begin{enumerate}
\item Let $\mathcal{F}_{1,d}$ be the functional on $D_{d-1},\mathcal{B}_{d-1}$ with 
$$\mathcal{F}_{1,d}(D_{d-1},\mathcal{B}_{d-1}): k \mapsto \mathcal{B}_{d-1}(3,D_{d-1}(k+1))+1.$$
\item Let $\mathcal{F}_{2,d}$ be the functional on $A_d$ with 
$$\mathcal{F}_{2,d}(A_d): (\eta,\delta) \mapsto |\delta|-1+A_d(\eta+3|\delta|).$$
\item Let $\mathcal{F}_{4,d}$ be the functional on $B_d$ with 
$$\mathcal{F}_{4,d}(B_d): (\eta,n) \mapsto B_d(\eta,n\tilde{e}_d).$$
\item Let $\mathcal{F}_{5,d}$ be the functional on $\mathcal{B}_d$ with 
$$\mathcal{F}_{5,d}(\mathcal{B}_d): k \mapsto d^{2^{\mathcal{B}_d(1,k)}}.$$
\end{enumerate}    
\end{definition}
\begin{theorem}
Let $d \geq 2$. Use the notations in the above definition and assume all the functions are ascending.
\begin{enumerate}
\item Assume $D_{d-1}$ satisfies $PD_{d-1}$ and $\mathcal{B}_{d-1}$ satisfies $PcB_{d-1}$. If $A_d=$

$\mathcal{F}_{1,d}(D_{d-1},\mathcal{B}_{d-1})$, then $A_d$ satisfies $PAa_d$.
\item Assume $A_d$ satisfies $PAa_d$. If $\bar{A}_d=\mathcal{F}_{2,d}(A_d)$, then $\bar{A}_d$ satisfies $PAb_d$.
\item Assume $B_d$ satisfies $PB_d$. If $\mathcal{B}_d=\mathcal{F}_{4,d}(B_d)$, then $\mathcal{B}_d$ satisfies $PcB_d$.
\end{enumerate}
\end{theorem}
\begin{proof}
(iii) follows since we assume $B_d$ is ascending, and $|\delta|=n$ implies $\delta_i\leq n$, so every component of $\delta$ is smaller than that component of $n\tilde{e}_d$. (i) and (ii) follow from the inductive proof of Theorem A(a) and Theorem A(b) in Section 4 of \cite{AH1}. Since the proofs of (i) and (ii) are more difficult, we include detailed proofs of (i) and (ii) for the reader's convenience.

To prove (i), we fix $\eta$, assume $F \in R=K[x_1,\ldots,x_N]$ is homogeneous of degree at most $d$ such that $\str(F) \geq \mathcal{B}_{d-1}(3,D_{d-1}(\eta+1))+1$. The Theorem F of \cite{AH1} says if $\deg(F)\leq d$ and $\str(F) \geq \mathcal{B}_{d-1}(3,h)+1$, then $\mathcal{D}F$ is not contained in an ideal generated by $h$-forms of degree at most $d-1$ where $\mathcal{D}F$ is the $K$-vector space generated by partial derivatives of $F$. Thus applying Theorem F to the situation of (i), we see $\mathcal{D}F$ is not contained in an ideal generated by $D_{d-1}(\eta+1)$-forms of degree at most $d-1$.

Let $(\mathcal{D}F)R$ be the $R$-ideal generated by $\mathcal{D}F$. If the height of $(\mathcal{D}F)R$ in $R$ is at most $\eta+1$, then there is a regular sequence $\underline{f}$ of length $(\eta+1)$ whose degrees are at most $d-1$, and $P\in \operatorname{Min}(\underline{f})$ such that $(\mathcal{D}F)R \subset P$. In this case the number of generators of $P$ is at most $D_{d-1}(\eta+1)$ by definition of property $PD_d$. This means $\mathcal{D}F$ is contained in an ideal generated by $D_{d-1}(\eta+1)$, which is a contradiction. This means the height of $(\mathcal{D}F)R$ in $R$ is at least $\eta+2$, so the height of $(\mathcal{D}F)R$ in $R/FR$ is at least $\eta+1$, which implies $R/FR$ is $(R_\eta)$, and we get (i).

To prove (ii), we choose a graded vector space $V \subset R$ with $\underline{\dim}V=\delta$ such that $\str(V)\geq |\delta|-1+A_d(\eta+3|\delta|)$, and prove any homogeneous $K$-basis of $V$ forms an $R_\eta$-sequence. We induct on $\dim V=|\delta|$. If $\dim V=1$, then the $K$-basis consists of a single element $F$, and $\str(V)=\str(F)\geq A_d(\eta+3|\delta|)$ implies $R/VR=R/FR$ is $R_{\eta+3|\delta|}$, hence is $R_\eta$. Suppose (ii) is proved for $\dim V=n-1$, and now assume $\dim V=n$. Choose any homogeneous $K$-basis $f_1,\ldots,f_n$. Let $V'$ be the graded $K$-vector space generated by $f_1,\ldots,f_{n-1}$, then $\str(V')\geq \str(V)$, so we apply the induction hypothesis to get $f_1,\ldots,f_{n-1}$ is an $R_\eta$-sequence. We see $f_n$ cannot lie in the ideal generated by $f_1,\ldots,f_{n-1}$, otherwise this is saying $\str(V)\leq \str(f_n)\leq n-1$, but by assumption $\str(V)>n-1$ which is a contradiction. In particular, $f_1,\ldots,f_{n-1}$ is a prime sequence and the image of $f_n$ is a nonzero element in the domain $R/(f_1,\ldots,f_{n-1})R$, so $f_1,\ldots,f_n$ is a regular sequence. Since $\str(V_j)\geq A_d(\eta+3|\delta|)$ for any $j$, for any $F\in V$ homogeneous, $\str(F)\geq A_d(\eta+3|\delta|)$, so by definition of $PAa_d$, $R/FR$ is $R_{\eta+3|\delta|}$, and height of $(\mathcal{D}F)R$ in $R$ is at least $\eta+3|\delta|+1$. Now we apply Theorem 2.5 of \cite{AH1}, which says the following: if $V$ is a graded vector space whose homogeneous $K$-basis forms a regular sequence, and for any homogeneous element $F \in V$, height of $(\mathcal{D}F)R$ is at least $\eta+3|\delta|+1$, then $R/VR$ is $R_\eta$. So $R/(f_1,\ldots,f_n)R$ is $R_\eta$. This means $f_1,\ldots,f_n$ is an $R_\eta$-sequence.
\end{proof}
We want to prove a similar result for $\mathcal{F}_{5,d}$. We will build a shortcut for this functional instead of directly applying the argument in \cite{AH1}. We start with the following lemma:
\begin{lemma}[Chardin \cite{Chardin}] \label{colon}
	Let $P\subset K[x_1,\dots,x_N]$ be a minimal prime of an ideal generated by a homogeneous regular sequence $f_1,\dots,f_k$ of degrees $d_1,\dots,d_k$. There exists a form $f$ of degree at most $d_1+\cdots+d_k-k$ such that
	\[
	P=(f_1,\dots,f_k):(f).
	\]
\end{lemma}

The next two lemmas justify the fact that we can choose $D_d(k)$ to be $(2d)^{2^{\prescript{1}{}{\mathcal{B}(k,d)}-1}}$. Assuming the number of variables is known, we show in Lemma \ref{boundcolon} how to bound the minimal number of generators of a particular kind of colon ideals as in Lemma \ref{colon}. To obtain a sharper bound, we extensively apply results and proofs of \cite{CS1} and \cite{CS2}.
\begin{lemma} \label{boundcolon}
	Let $I\subset K[x_1,\dots,x_{B+1}]$ be an ideal generated by a regular sequence of $c$ forms of degree at most $d$, and $f$ be a form of degree at most $cd-c$. Then the minimal number of generators of $I:(f)$ is bounded by
	\[
	\beta_0(I:(f))\leq B(d+1)(2d)\prod_{i=3}^{B}((d^2+2d-1)^{2^{i-3}}+1)+1.
	\]
	If we further assume $B,d\geq 3$ or $B\geq 4$, then we have the additional inequality
	\[
	\beta_0(I:(f)) \leq (2d)^{2^{B-1}}.
	\]
\end{lemma}

\begin{proof}
	If $c=1$, then $\beta_0(I:(f))= 1$, so assume $c\geq 2$. After a faithfully flat base change, we may assume $K$ is infinite. Denote $R:=K[x_1,\dots,x_{B+1}]$ and let $f_1,\dots,f_c$ be the regular sequence with deg$(f_i)\leq d$. Consider the exact sequence
	\begin{equation*}
		0\to R/(I:(f))\;{\xrightarrow {\ f\ }}\;R/I\;{\xrightarrow {\ \ }}\;R/(I+(f))\to 0.
	\end{equation*}
	Since $c\leq \beta_0(I+(f))\leq c+1$, using the long exact sequence of Tor$_i^{R}(-,K)$ induced from the above short exact sequence, we get $\beta_0(I:(f))\leq \beta_1(I+(f))+1$.
	
	With the notations of Section 2, let $J:=\text{Gin}_0(I+(f))$. Denote ${R}_{[i]}:=K[x_1,\dots,x_i]$. Let $(I+(f))_{\langle i \rangle}$ denote the image of $I+(f)$ in $R/(l_{B+1},\dots,l_{i+1})\cong R_{[i]}$ where $l_{B+1},\dots,l_{i+1}$ are general linear forms, and let $m_{[i]}$ denote the homogeneous maximal ideal of $R_{[i]}$. Let $J_{[i]}$ denote $J\cap R_{[i]}$. By \cite[Proposition 2.2]{AH2}, $J$ is strongly stable with $\beta_1(I+(f))\leq \beta_1(J)$. So by \cite{EK} and \cite[Proposition 1.6]{CS1}, $\beta_1(J)=\sum_{u\in G(J)}(m(u)-1)\leq B\cdot \lvert G(J) \rvert\leq B\prod_{i=1}^{B}(D(J_{[i]})+1)$. Using \cite[Theorem 2.20]{CS2}, we can get $D(J_{[i]})\leq \reg((I+(f))_{\langle i \rangle})$ for all $i$. Notice that $\reg((I+(f))_{\langle i \rangle})\leq id-i+1$ for all $i\leq c$, because $m_{[i]}^{id-i+1}\subseteq (I+(f))_{\langle i \rangle}$.
	
	To bound $\reg((I+(f))_{\langle i \rangle})$ for $i\geq c+1$, we follow the proof of \cite[Theorem 2.4 and Corollary 2.6]{CS1}. Let $\lambda(M)$ denote the length of an Artinian module $M$. Using the same proof of \cite[Theorem 2.4]{CS1}, we can get
    \begin{equation}\label{regrecur}
    \begin{split}
    \reg((I+(f))_{\langle i \rangle})\\
    \leq \text{max}\{d,cd-c,\reg((I+(f))_{\langle i-1 \rangle})\}+\lambda\left(\frac{R_{[c]}}{(I+(f))_{\langle c \rangle}}\right)\prod_{j=c+2}^{i} \reg((I+(f))_{\langle j-1 \rangle})\\
    \leq \text{max}\{d,cd-c,\reg((I+(f))_{\langle i-1 \rangle})\}+d^c\prod_{j=c+1}^{i-1} \reg((I+(f))_{\langle j \rangle}).    
    \end{split}
    \end{equation}
	The last inequality holds since $R_{[c]}/(I+(f))_{\left<c\right>}$ is a quotient ring of $R_{[c]}/(g_1,\dots,g_c)$, where $g_1,\dots,g_c$ are the images of $f_1,\dots,f_c$ in $R_{[c]}$ and form a regular sequence with deg$(g_i)\leq d$.
	
	Now we use (\ref{regrecur}) recursively to bound $\reg((I+(f))_{\langle i \rangle})$ for $i\geq c+1$. Set $B_0:=cd-c+1$, and recall that $B_0$ bounds $\reg((I+(f))_{\langle c \rangle})$. Apply (\ref{regrecur}) to $(I+(f))_{\langle c+1 \rangle}$ to get $\reg((I+(f))_{\langle c+1 \rangle})\leq cd-c+1+d^c=:B_1$. For $j\geq2$, we set $B_j:=B_{j-1}+d^c\prod_{k=1}^{j-1} B_k  \leq (B_{j-1})^2 \leq (B_1)^{2^{j-1}}$. Hence for all $i\geq c+1$, $\reg((I+(f))_{\langle i \rangle})\leq B_{i-c}\leq (cd-c+1+d^c)^{2^{i-c-1}}\leq (d^2+2d-1)^{2^{i-3}}$, where the last inequality holds since the second last bound is decreasing as a function of $c$ and $c\geq2$.
	
	If $B\geq 4$, then $B\leq 2^{B-2}$. Combining all the previous inequalities, we get
	\begin{equation*}
		\begin{split}
			\beta_0(I:(f))
			&\leq B(d+1)(2d)\prod_{i=3}^{B}((d^2+2d-1)^{2^{i-3}}+1)+1\\
			&\leq 2^{B-2}(d+1)(2d)\prod_{i=3}^{B}(d^2+2d)^{2^{i-3}} \leq (2d)(2d)\prod_{i=3}^{B}(2d^2+4d)^{2^{i-3}} \\
			&\leq (2d)^2\prod_{i=3}^{B}(2d)^{2^{i-2}}=(2d)^{2^{B-1}}.
		\end{split}
	\end{equation*}
	If $B=3$ and $d\geq 3$, one easily checks that $\beta_0(I:(f))\leq 3(d+1)(2d)(d^2+2d)+1\leq (2d)^4$.
\end{proof}

Given a function $\mathcal{B}_d(1,k)$ satisfying $PcB_d$, Lemma \ref{boundP} constructs our value $(2d)^{2^{\mathcal{B}_d(1,k)}-1}$ for the bound $D_d(k)$ satisfying $PD_d$ by passing to a polynomial subring with at most $\mathcal{B}_d(1,k)+1$ many variables first, using the result of Lemma \ref{colon} and Lemma \ref{boundcolon}.
\begin{lemma} \label{boundP}
Let $K$ be an algebraically closed field, $P\subset K[x_1,\dots,x_N]$ be a minimal prime of an ideal generated by a regular sequence of $k$ or fewer forms of degree at most $d$. Assume the function $\mathcal{B}_d(1,\cdot)$ satisfies $PcB_d$ for $\eta=1$, and let $B=\mathcal{B}_d(1,k)$ be its value at $(1,k)$. Then the minimal number of generators of $P$ is bounded by
\[
\beta_0(P)\leq (2d)^{2^{B-1}}.
\]
If $d \geq 2$, then
\[
\beta_0(P)\leq d^{2^B}.
\]
 
\end{lemma}

\begin{proof}
	Let $f_1,\dots,f_c$ be the regular sequence with $c\leq k$ and deg$(f_i)\leq d$, let $I$ be the ideal it generates. By assumption on $B$, there exists a prime sequence $G_1,\dots,G_s$ with $s\leq B$ such that $f_1,\dots,f_c\in K[G_1,\dots,G_s]$. Denote $R=K[x_1,\dots,x_N]$ and $S=K[G_1,\dots,G_s]$. Then pd$_R$$(R/I)=c\leq s$ and pd$_S$$(S/P\cap S)\leq s$. Notice that $R$ is a free and thus faithfully flat module over $S$ since we can extend $G_1,\dots,G_s$ to a maximal regular sequence $G_1,\dots,G_N$ in $R$ to get free extensions $K[G_1,\dots, G_s]\xhookrightarrow{}K[G_1,\dots, G_N]$ and $K[G_1,\dots, G_N]\xhookrightarrow{}K[x_1,\dots, x_N]$. Consequently we get pd$_R$$(R/P)\leq s$ once we have shown $P=(P\cap S)R$. By faithfully flatness $f_1,\dots,f_c\in P\cap S$ remains a regular sequence in $S$ and so $c=$ ht$\,P\geq$ ht$\,(P\cap S)R=$ ht$\,P\cap S\geq c$. Now by \cite[Corollary 2.9]{AH1}, $(P\cap S)R$ is a prime ideal, therefore $P=(P\cap S)R$.
	
	By Lemma \ref{colon}, there exists a form $f\in R$ of degree at most $cd-c$ such that $P=I:(f)$. Consider the exact sequence
	\begin{equation}\label{exact1}
		0\to R/P\;{\xrightarrow {\ f\ }}\;R/I\;{\xrightarrow {\ \ }}\;R/(I+(f))\to 0.
	\end{equation}
	
	It follows that pd$_R$$(R/(I+(f)))\leq s+1$. Then depth$\,R/(I+(f))\geq N-(s+1)$ by the Auslander-Buchsbaum formula. Let $l_{s+2},\dots,l_N\in R$ be a sequence of linear forms regular on $R/P$, $R/(I+(f))$, and $R/(f_1,\dots,f_c)$. Fix a graded isomorphism from $R/(l_{s+2},\dots,l_N)$ to $K[x_1,\dots,x_{s+1}]$, let $\overline{R}$ denote $K[x_1,\dots,x_{s+1}]$ and ``$\overline{\phantom{,,}}$'' denote the image of polynomials or ideals of $R$ in $\overline{R}$. Notice that $\beta_0(P)=\beta_0(\overline{P})$. Since $l_{s+2},\dots,l_N$ is a regular sequence on $R/(I+(f))$, the short exact sequence in (\ref{exact1}) remains exact after tensoring with $\overline{R}$. It follows that $\overline{P}=\overline{I}:(\overline{f})$. Notice that $\overline{f_1},\dots,\overline{f_c}$ is a regular sequence in $\overline{R}$, so we can apply Lemma \ref{boundcolon} to get $\beta_0(\overline{P})\leq (2d)^{2^{B-1}}$ for $B\geq 4$, or $B=3$ and $d\geq 3$.
	
	If $B=2$, it is clear that $\beta_0(P)=\beta_0(P\cap S)\leq 2 \leq (2d)^{2^1}$. If $B=3$ and $d=2$, then $P$ is a minimal prime of an ideal generated by two quadrics. Let $e(R/P)$ denote the Hilbert-Samuel multiplicity of $R/P$ with respect to the maximal ideal $(x_1,\dots,x_N)$, we have $e(R/P)\leq 4$. If either $P$ contains a linear form or $e(R/P)=4$, then $P$ is a complete intersection and so $\beta_0(P)= 2$. Otherwise $P$ contains no linear forms and $e(R/P)=3=1+\het(P)$, so we can apply \cite[Theorem 4.2]{EISENBUD198489} to see that $R/P$ is Cohen-Macaulay and hence $\beta_0(P)\leq e(R/P)=3$. In both cases $\beta_0(P) \leq 3<4^{2^2}$, hence the first inequality is proved. For the second inequality, note that $d \geq 2$ implies $2d \leq d^2$, therefore
$$(2d)^{2^{B-1}} \leq (d^2)^{2^{B-1}}=d^{2^B}.$$
\end{proof}
The above lemma immediately implies:
\begin{theorem}
Assume $\mathcal{B}_d$ satisfies $PcB_d$. If $D_d=\mathcal{F}_{5,d}(\mathcal{B}_d)$, then $D_d$ satisfies $PD_d$.    
\end{theorem}
Next we introduce the $F$-decomposition.
\begin{definition}[Decompositions of dimension vectors]
We fix $d \in \mathbb{Z}_+$ and a function $F:E_d=\mathbb{N}^d \to \mathbb{N}$. Let $\delta \in E_d$. If $\delta_i \neq 0$, we say a simple $F$-decomposition of $\delta$ of degree $i$ is $\delta-e_i+F(\delta)\tilde{e}_{i-1}=\delta-e_i+F(\delta)(e_{i-1}+\ldots+e_1)$. An $F$-decomposition is a composition of simple decompositions. We write $\delta \rightarrowtail \delta'$ if $\delta'$ is a decomposition of $\delta$, write $\delta \rightarrowtail\leq \delta''$ if there exists $\delta'$ such that $\delta \rightarrowtail\delta'$ and $\delta' \leq \delta''$. 
\end{definition}
\begin{lemma}
For a function $F:\mathbb{N}^d \to \mathbb{Z}_+$, define the function $B_d(\cdot)$ in the following way: for every $\delta \in \mathbb{N}^d$, there are only finitely many $F$-decompositions of $\delta$, and $B_d(\delta)=\max \{n:\delta \rightarrowtail ne_1\}$. Then $B_d$ is well-defined and uniquely determined by $F$, hence there is a functional $\mathcal{F}_{dec,d}$ such that $B_d=\mathcal{F}_{dec,d}(F)$. Moreover, for such $B_d$ we have $B_d(\delta)\geq|\delta|$.   
\end{lemma}
\begin{proof}
We consider the lexicographic order on $E$; we say $\delta >_{lex} \delta'$ if in the largest degree $d'$ where $\delta$ and $\delta'$ differ, $\delta_{d'}>\delta'_{d'}$. For each $d \geq 1$, $>_{lex}$ is a well-ordering on $E_d$, that is, every nonempty set has a minimal element. If $\delta \in E_1$, then its only decomposition is itself. If $\delta \notin E_1$, then it has at most $\deg \delta-1$ simple decompositions which are smaller under $>_{lex}$. If all these simple decompositions have finitely many decompositions, then $\delta$ itself has finitely many decompositions. Therefore, the set of dimension sequences in $E_d$ with infinitely many decompositions must be empty, otherwise it has a minimal element $\delta_0$ since $>_{lex}$ is a well order. However, that all simple decompositions of $\delta_0$ have finitely many decomposition implies that $\delta_0$ has finitely many decomposition, which is a contradiction. We see every decomposition terminates at $E_1$. Thus for every nonzero $\delta \in \mathbb{N}^d$, the set $ \{n:\delta \rightarrowtail ne_1\}$ is finite and nonempty, so $B_d(\eta,\delta)$ is well-defined.
\end{proof}
\begin{lemma}
Fix $d \geq 2$. Let $F$ be a function $F:\mathbb{N}^d \to \mathbb{Z}_+$. Consider a function $B_d:\mathbb{N}^d\to\mathbb{Z}_+$ satisfying: $B_d(ne_1)=n$, and $B_d(\delta)$ is equal to the maximum of $|\delta|$ and $B_d(\delta')$ where $\delta'$ runs through all simple 
$F(\cdot)$-decompositions of $\delta$. Then $B_d=\mathcal{F}_{dec,d}(F)$.
\end{lemma}
\begin{proof}
This equality is trivially true for $\delta=ne_1$. Now we use induction on $\delta$ with respect to $>=>_{lex}$ which is a well-ordering. In this case we only need to prove that this equality holds for any $\delta'<\delta$, then it holds for $\delta$. Now we assume the equality holds for any $\delta'<\delta$. We may also assume $\delta \neq |\delta|e_1$, so $\delta$ has at least one decomposition. In this case we have
$$\{n:\delta \rightarrowtail ne_1\}=\bigcup_{\delta \rightarrowtail \delta' \textup{ simply}}\{n:\delta'\rightarrowtail ne_1\} \neq \emptyset.$$
Also by induction $B_d(\delta')\geq |\delta'|\geq |\delta|$, so
\begin{align*}
B_d(\delta)=\max\{|\delta|,B_d(\delta'),\delta \rightarrowtail\delta' \textup{simply}\}\\
=\max\{B_d(\delta'),\delta \rightarrowtail\delta' \textup{simply}\}\\
=\max_{\delta \rightarrowtail\delta' \textup{simply}}\{\max\{n:\delta'\rightarrowtail ne_1\}\}\\
=\max \bigcup_{\delta \rightarrowtail \delta' \textup{ simply}}\{n:\delta'\rightarrowtail ne_1\}\\
=\max \{n:\delta \rightarrowtail ne_1\}.
\end{align*}
\end{proof}
\begin{lemma}
Let $\bar{A}_d$ be a function $\bar{A}_d(\cdot,\cdot):\mathbb{Z}_+\times\mathbb{N}^d \to \mathbb{Z}_+$. For every $\eta \in \mathbb{Z}_+$, define the function $B_d(\eta,\delta)=\mathcal{F}_{dec,d}(2\bar{A}_d(\eta,\cdot))(\delta)$. Then $B_d$ is well-defined and uniquely determined by $\bar{A}_d$, hence there is a functional $\mathcal{F}_{3,d}$ such that $B_d=\mathcal{F}_{3,d}(\bar{A}_d)$. Moreover for such $B_d$ we have $B_d(\eta,\delta)\geq|\delta|$.   
\end{lemma}
\begin{lemma}\label{3.11}
Let $\bar{A}_d$ be a function $\bar{A}_d(\cdot,\cdot):\mathbb{Z}_+\times\mathbb{N}^d \to \mathbb{Z}_+$ and $B_d=\mathcal{F}_{3,d}(\bar{A}_d)$. Then $\bar{A}_d(\cdot,\cdot)$ satisfying $PAb_d$ implies that $B_d$ satisfies $PB_d$.   
\end{lemma}
\begin{proof}
The proof lies in the part of Section 4 of \cite{AH1} from Theorem A to Theorem B. Roughly speaking, if we have a graded vector space, then either its strength is large enough so that a $K$-basis forms an $R_\eta$-sequence, or its strength is small, so we can replace basis elements with elements of lower degree until the strength of the new set of elements is large enough. In this case the dimension vector $\delta$ is replaced with a $2\bar{A}_d$-decomposition $\delta'$, and the length of the $R_\eta$-sequence is equal to $|\delta'|$, which is smaller than or equal to $\max\{n:\delta\rightarrowtail ne_1\}$. This is just the value of $B_d=\mathcal{F}_{3,d}(\bar{A}_d)$. So, we are done.    
\end{proof}
\begin{lemma}
An initial condition $\mathcal{I}_\alpha$ satisfying the AH-condition is given by
$$A_1(\eta)=1,\bar{A}_1(\eta,ne_1)=1,B_1(\eta,ne_1)=n,\mathcal{B}_1(\eta,n)=n, D_1(k)=k.$$  
\end{lemma}
The above lemma is easy to check. Gathering all the results above, we get:
\begin{theorem}
The recurrence relation $\mathcal{R}_\alpha$ given by $\mathcal{F}_{i,j}, 1 \leq i \leq 5, j \geq 2$ is an AH recurrence relation. The package $\mathcal{P}_\alpha$ determined by relation $\mathcal{R}_\alpha$ and initial condition $\mathcal{I}_\alpha$ is an AH package.    
\end{theorem}

\section{Reduction Steps}
In this section, we will bound the package $\mathcal{P}_\alpha$ using other packages.

\subsection{Reduced AH recurrence relations}
In general, if the functionals in a recurrence relation depend on very few inputs, we can compose the functionals. This gives us a recurrence relation of smaller size and packages of smaller size, which is more convenient to use. For example, the recurrence relation $\mathcal{R}_\alpha$ looks like:
\begin{center}
\begin{tikzcd}[column sep=small,row sep=large]
& & &\mathcal{B}_{d-1}\in PcB_{d-1}\arrow[r,blue,"\mathcal{F}_{5,d-1}",Rightarrow]\arrow[llld,blue,"\mathcal{F}_{1,d}"',Rightarrow]&D_{d-1} \in PD_{d-1}\arrow[lllld,blue,"\mathcal{F}_{1,d}",Rightarrow]\\
A_d \in PAa_d \arrow[r,blue,"\mathcal{F}_{2,d}"',Rightarrow] & \bar{A}_d \in PAb_d\arrow[r,red,"\mathcal{F}_{3,d}",Rightarrow]& B_d \in PB_d\arrow[r,blue,"\mathcal{F}_{4,d}",Rightarrow]&\mathcal{B}_d\in PcB_d\arrow[r,blue,"\mathcal{F}_{5,d}",Rightarrow]\arrow[llld,blue,"\mathcal{F}_{1,d+1}"',Rightarrow]&D_d \in PD_d\arrow[lllld,blue,"\mathcal{F}_{1,d+1}",Rightarrow]\\
A_{d+1} \in PAa_{d+1} & & & &
\end{tikzcd}    
\end{center}

Among these functionals, only $\mathcal{F}_{3,d}$ is not explicit using elementary operations like addition, multiplication, and taking exponents. Thus we only need to look at two sequences of functions $\bar{A}_d$ and $B_d$, and compose all the blue functionals in the above diagram. To be precise, for $d \geq 2$, we consider the recurrence relation $\mathcal{R}_\beta$ given by 
$$B_d=\mathcal{G}_{1,d}(\bar{A}_d), \bar{A}_{d+1}=\mathcal{G}_{2,d}(B_d),$$
where
$$\mathcal{G}_{1,d}=\mathcal{F}_{3,d},\mathcal{G}_{2,d}=\mathcal{F}_{2,d+1}\mathcal{F}_{1,d+1}(\mathcal{F}_{5,d}\mathcal{F}_{4,d},\mathcal{F}_{4,d}).$$
The diagram looks like:
\begin{center}
\begin{tikzcd}
\bar{A}_d \arrow[r,"\mathcal{G}_{1,d}",Rightarrow] & B_d\arrow[ld,"\mathcal{G}_{2,d}",Rightarrow]\\
\bar{A}_{d+1} \arrow[r,"\mathcal{G}_{1,d+1}",Rightarrow] & B_{d+1}
\end{tikzcd}    
\end{center}
An initial condition can be derived from $\mathcal{I}_\alpha$. That is, we set
\begin{align*}
\mathcal{B}_1(\eta,n)=n,D_1(k)=k\\
A_2(k)=\mathcal{F}_{1,2}(\mathcal{B}_1,D_1)(k)=\mathcal{B}_1(3,D_1(k+1))+1=k+2\\
\bar{A}_2(\eta,\delta)=\mathcal{F}_{2,2}(A_2)(\eta,\delta)=|\delta|-1+A_2(\eta+3|\delta|)=\eta+4|\delta|+1.
\end{align*}
We will call the condition $\bar{A}_2(\eta,\delta)=\eta+4|\delta|+1$ the initial condition $\mathcal{I}_\beta$.

The following theorem is automatic by definition.
\begin{theorem}
Let $\mathcal{P}_\beta$ be a 2-package of functions $\bar{A}_d,B_d,d \geq 2$ where $\bar{A}_d:\mathbb{Z}_+ \times \mathbb{N}^d \to \mathbb{Z}_+$, $B_d:\mathbb{Z}_+ \times \mathbb{N}^d \to \mathbb{Z}_+$ given by initial condition $\mathcal{I}_\beta$ and recurrence relation $\mathcal{R}_\beta$. Then the functions in the 2-package $\mathcal{P}_\beta$ also appear in the 5-package $\mathcal{P}_\alpha$, so $B_d$ satisfies $PB_d$ and $\bar{A}_d$ satisfies $PAb_d$.   
\end{theorem}
Here $B_d(\cdot,\cdot)=\mathcal{G}_{1,d}(\bar{A}_d)=\max\{n:\delta \rightarrowtail ne_1\}$ where $\rightarrowtail$ refers to $2\bar{A}_d(\eta,\cdot)$-decomposition. For each fixed $\eta$, $B_d(\eta,\delta)$ only depends on $\bar{A}_d(\eta,\cdot)$. $\mathcal{G}_{2,d}$ is explicit and computable, given by:
$$\bar{A}_{d+1}(\eta,\delta)=\mathcal{G}_{2,d}(B_d)=|\delta|+B_d(3,d^{2^{B_d(1,(\eta+3|\delta|+1)\tilde{e}_d)}}\tilde{e}_d).$$
In particular, we have
$$\bar{A}_{d+1}(\eta,\delta)\leq|\delta|+B_d(3,d^{2^{B_d(3,(\eta+3|\delta|+1)\tilde{e}_d)}}\tilde{e}_d)$$
because $B_d$ is ascending.

\subsection{Restricted to $\eta=3$ case}
We notice that the induction process can be restricted to the case where $\eta=3$. We define the functor $\mathcal{G}_{3,d}$ and $\mathcal{G}_{4,d}$ as follows. If $\bar{A}_d:\mathbb{N}^d \to \mathbb{Z}_+$, then
$$\mathcal{G}_{3,d}(\bar{A}_d):\delta \mapsto \max\{n: \delta \rightarrowtail ne_1\}$$
where $\rightarrowtail$ refers to $2\bar{A}_d$-decomposition, so $\mathcal{G}_{3,d}(\bar{A}_d)$ is also a function $\mathbb{N}^d \to \mathbb{Z}_+$. That is, $\mathcal{G}_{3,d}=\mathcal{F}_{dec,d}(2\cdot)$. Now let $B_d:\mathbb{N}^d \to \mathbb{Z}_+$, define
$$\mathcal{G}_{4,d}(B_d):\delta \mapsto |\delta|+B_d(d^{2^{B_d((3|\delta|+4)\tilde{e}_d)}}\tilde{e}_d).$$
$\mathcal{G}_{4,d}(B_d)$ is also a function $\mathbb{N}^d \to \mathbb{Z}_+$.
Thus we have:
\begin{theorem}
Let $\mathcal{P}_\gamma$ be a 2-package of functions $\bar{A}_d,B_d,d \geq 2$ where $\bar{A}_d:\mathbb{N}^d \to \mathbb{Z}_+$, $B_d:\mathbb{N}^d \to \mathbb{Z}_+$ given by initial conditions $\mathcal{I}_\gamma$: $\bar{A}_2(\delta)=4|\delta|+4$ and recurrence relation $\mathcal{R}_\gamma$: $B_d(\delta)=\mathcal{G}_{3,d}(\bar{A}_d)$ and $\bar{A}_{d+1}=\mathcal{G}_{4,d}(B_d)$. Then for any $d$ we have $B_{\mathcal{P}_\gamma,d}(\delta) \geq B_{\mathcal{P}_\beta,d}(3,\delta)$ and $\bar{A}_{\mathcal{P}_\gamma,d}(\delta) \geq \bar{A}_{\mathcal{P}_\beta,d}(3,\delta)$. In particular, the functions $\bar{A}_d$ and $B_d$ satisfy $PAb_d$ and $PB_d$ respectively for $\eta=3$. 
\end{theorem}
\subsection{Bounding $F$-decompositions}
We fix $d \geq 1$ and an ascending positive function $F:E_d=\mathbb{N}^d \to \mathbb{Z}_+$. In this subsection, we consider bounds on the decompositions of a dimension vector $\delta \in E_d$, and build a bound as a functional of $F$.

Let $\sigma \in \mathbb{Z}_+$. We start with dimension vector $\delta=(\sigma_d,\ldots,\sigma_1)=\sum_{1 \leq i \leq d}\sigma_ie_i$ for $\sigma_i \leq \sigma$ for all $1 \leq i \leq d$.

If we start with only simple decompositions in degree 2, then the first step is:
$$(\sigma_d,\ldots,\sigma_1)\rightarrowtail (\sigma_d,\ldots,\sigma_3,\sigma_2-1,\sigma_1+F(\sigma_d,\ldots,\sigma_2,\sigma_1)).$$
The second step is
\begin{align*}
(\sigma_d,\ldots,\sigma_2-1,\sigma_1+F(\sigma_d,\ldots,\sigma_2,\sigma_1))\\
\rightarrowtail (\sigma_d,\ldots,\sigma_2-2,\sigma_1+F(\sigma_d,\ldots,\sigma_2,\sigma_1)+F((\sigma_d,\ldots,\sigma_2-1,\sigma_1+F(\sigma_d,\ldots,\sigma_2,\sigma_1)))).
\end{align*}
The third step is
\begin{align*}
\rightarrowtail (\sigma_d,\ldots,\sigma_2-3,\sigma_1+F(\sigma_d,\ldots,\sigma_2,\sigma_1)+F((\sigma_d,\ldots,\sigma_2-1,\sigma_1+F(\sigma_d,\ldots,\sigma_2,\sigma_1)))+\\
F(\sigma_d,\ldots,\sigma_2-2,\sigma_1+F(\sigma_d,\ldots,\sigma_2,\sigma_1)+F((\sigma_d,\ldots,\sigma_2-1,\sigma_1+F(\sigma_d,\ldots,\sigma_2,\sigma_1))))),
\end{align*}
and so on. Since $F$ is ascending, after replacing every element $\sigma_2-i$ under $F$ by $\sigma$ and $\sigma_i,i \geq 3$ by $\sigma$, the corresponding vector will be larger. Thus, we denote 
$$\prescript{d}{}{\psi_{1,\sigma}}(\epsilon)=\epsilon+F(\sigma,\ldots,\sigma,\epsilon).$$ 
Since $F>0$, $\prescript{d}{}{\psi_{1,\sigma}}(\epsilon)>\epsilon$. In particular, $\prescript{d}{}{\psi^n_{1,\sigma}}(\epsilon) \geq \epsilon+n$. Then the vectors in the above steps are smaller than the following vectors:
\begin{align*}
(\sigma_d,\ldots,\sigma_2,\sigma_1) \rightarrowtail\leq (\sigma_d,\ldots,\sigma_2-1,\prescript{d}{}{\psi_{1,\sigma}}(\sigma_1))\\
\rightarrowtail\leq (\sigma_d,\ldots,\sigma_2-2,\prescript{d}{}{\psi^2_{1,\sigma}}(\sigma_1)) \rightarrowtail\leq (\sigma_d,\ldots,\sigma_2-3,\prescript{d}{}{\psi^3_{1,\sigma}}(\sigma_1))\\
\ldots \rightarrowtail\leq (\sigma_d,\ldots,\sigma_3,0,\prescript{d}{}{\psi^{\sigma_2}_{1,\sigma}}(\sigma_1)) \rightarrowtail\leq (\sigma_d,\ldots,\sigma_3,0,\prescript{d}{}{\psi^{\sigma}_{1,\sigma}}(\sigma)).
\end{align*}
In this process we see that the $e_1$-coordinate of the dimension vector increases and the $e_2$-coordinate decreases. Also, $\sigma_2\leq \sigma<\prescript{d}{}{\psi^{\sigma}_{1,\sigma}}(\sigma).$ Thus if $(\sigma_d,\ldots,\sigma_1) \rightarrowtail \delta'$ by decompositions in degree 2 and $(\sigma_d,\ldots,\sigma_1) \leq \sigma\tilde{e}_d$, then $\delta' \leq (\sigma_d,\ldots,\sigma_3,\prescript{d}{}{\psi^{\sigma}_{1,\sigma}}(\sigma),\prescript{d}{}{\psi^{\sigma}_{1,\sigma}}(\sigma))$.

Now assume there is a decomposition in degree 3. We consider two cases. The first case: this is the first decomposition of $\delta$. Then 
\begin{equation*}
(\sigma_d,\ldots,\sigma_1)\rightarrowtail (\sigma_d,\ldots,\sigma_3-1,\sigma_2+F(\sigma_d,\ldots,\sigma_1),\sigma_1+F(\sigma_d,\ldots,\sigma_1)). \tag{*}    
\end{equation*}
The second case: there are some decompositions in degree 2 before this decomposition. No matter what decompositions they are, the resulting vector is always smaller than or equal to $(\sigma_d,\ldots,\sigma_3,\prescript{d}{}{\psi^{\sigma}_{1,\sigma}}(\sigma),\prescript{d}{}{\psi^{\sigma}_{1,\sigma}}(\sigma))$, and then we are forced to decompose it in degree 3. Therefore, we get
\begin{align*}
(\sigma_d,\ldots,\sigma_1)\rightarrowtail\leq \\
\sigma_de_d+\ldots+(\sigma_3-1)e_3+(\prescript{d}{}{\psi^{\sigma}_{1,\sigma}}(\sigma)+F(\sigma_d,\ldots,\sigma_3,\prescript{d}{}{\psi^{\sigma}_{1,\sigma}}(\sigma),\prescript{d}{}{\psi^{\sigma}_{1,\sigma}}(\sigma)))(e_2+e_1).  \tag{**} 
\end{align*}
The vector on the second line of $(**)$ is always bigger than the right half of $(*)$. So we can just take the second line of $(**)$ as an upper bound for the decomposition in degree 3. Then we can iterate it again. In each iteration, we replace all $\sigma_3-i$ under $F$ by $\sigma$, and all $\sigma_i,i \geq 4$ by $\sigma$. We denote
$$\prescript{d}{}{\psi_{2,\sigma}}(\epsilon)=\prescript{d}{}{\psi^{\epsilon}_{1,\epsilon}}(\epsilon)+F(\sigma,\ldots,\sigma,\prescript{d}{}{\psi^{\epsilon}_{1,\epsilon}}(\epsilon),\prescript{d}{}{\psi^{\epsilon}_{1,\epsilon}}(\epsilon)).$$
Then
\begin{align*}
(\sigma_d,\ldots,\sigma_2,\sigma_1) \rightarrowtail\leq (\sigma_d,\ldots,\sigma_3-1,\prescript{d}{}{\psi_{2,\sigma}}(\sigma),\prescript{d}{}{\psi_{2,\sigma}}(\sigma))\\
\rightarrowtail\leq (\sigma_d,\ldots,\sigma_3-2,\prescript{d}{}{\psi^2_{2,\sigma}}(\sigma),\prescript{d}{}{\psi^2_{2,\sigma}}(\sigma)) \rightarrowtail\leq \ldots\\
\rightarrowtail\leq (\sigma_d,\ldots,\sigma_4,0,\prescript{d}{}{\psi^{\sigma_3}_{2,\sigma}}(\sigma),\prescript{d}{}{\psi^{\sigma_3}_{2,\sigma}}(\sigma)) \rightarrowtail\leq (\sigma_d,\ldots,\sigma_4,0,\prescript{d}{}{\psi^{\sigma}_{2,\sigma}}(\sigma),\prescript{d}{}{\psi^{\sigma}_{2,\sigma}}(\sigma)). 
\end{align*}
Here in the second step we can run the same argument because $\sigma<\prescript{d}{}{\psi_{2,\sigma}}(\sigma)$, so we can replace $\sigma$ by $\prescript{d}{}{\psi_{2,\sigma}}(\sigma)$ and since $\sigma_i<\sigma<\prescript{d}{}{\psi_{2,\sigma}}(\sigma)$ for all $i$, we can apply the result for decompositions in dimension 2.

Therefore, for $1 \leq i \leq d-1$, we can define $\prescript{d}{}{\psi_{i,\sigma}}(\epsilon)$ recursively from $F$:
$$\prescript{d}{}{\psi_{i,\sigma}}(\epsilon)=\prescript{d}{}{\psi^{\epsilon}_{i-1,\epsilon}}(\epsilon)+F(\sigma(\tilde{e}_d-\tilde{e}_i)+\prescript{d}{}{\psi^{\epsilon}_{i-1,\epsilon}}(\epsilon)\tilde{e}_i),$$

then
$$(\sigma_d,\ldots,\sigma_2,\sigma_1)\rightarrowtail\leq(\sigma_d,\ldots,\sigma_{i+2},0,\prescript{d}{}{\psi^{\sigma}_{i,\sigma}}(\sigma),\ldots,\prescript{d}{}{\psi^{\sigma}_{i,\sigma}}(\sigma)).$$
In particular, for $i=d-1$,
$$(\sigma_d,\ldots,\sigma_2,\sigma_1)\rightarrowtail\leq(0,\prescript{d}{}{\psi^{\sigma}_{d-1,\sigma}}(\sigma),\ldots,\prescript{d}{}{\psi^{\sigma}_{d-1,\sigma}}(\sigma)).$$
Let $\prescript{d}{}{\phi_i}(\sigma)=\prescript{d}{}{\psi^{\sigma}_{i,\sigma}}(\sigma)$. Then
\begin{align*}
(\sigma_d,\ldots,\sigma_2,\sigma_1)\\
\rightarrowtail\leq (0,\prescript{d}{}{\phi_{d-1}}(\sigma),\ldots,\prescript{d}{}{\phi_{d-1}}(\sigma))\\
\rightarrowtail\leq (0,0,\prescript{d}{}{\phi_{d-2}}\prescript{d}{}{\phi_{d-1}}(\sigma),\ldots,\prescript{d}{}{\phi_{d-2}}\prescript{d}{}{\phi_{d-1}}(\sigma))\\
\rightarrowtail\leq\ldots\rightarrowtail\leq (0,\ldots,0,\prescript{d}{}{\phi_1}\ldots\prescript{d}{}{\phi_{d-1}}(\sigma)).
\end{align*}
\begin{definition}
The functional from $F:E_d \to \mathbb{Z}_+$ to the function $\prescript{d}{}{\phi_1}\ldots\prescript{d}{}{\phi_{d-1}}:\sigma \to \prescript{d}{}{\phi_1}\ldots\prescript{d}{}{\phi_{d-1}}(\sigma)$ makes sense, so we denote it by $\mathcal{G}_{5,d}$. Its output is a function $\mathbb{Z}_+ \to \mathbb{Z}_+$. 
\end{definition}
The arguments above tell us 
$$\mathcal{G}_{5,d}(F)(\sigma) \geq \mathcal{F}_{dec,d}(F)(\sigma\tilde{e}_d).$$
\begin{remark}
The notations $\prescript{d}{}{\psi}_{i,\sigma}(\cdot)$ and $\prescript{d}{}{\phi}_i(\cdot)$ depend on the function $F:E_d \to \mathbb{Z}_+$. We will point out $F$ when we use these notations.    
\end{remark}
It is easy to see the functional $\mathcal{G}_{5,d}$ is ascending in terms of $F$. The above arguments in Subsection 4.3 imply:
\begin{proposition}
If $B_d=\mathcal{G}_{3,d}(F)$, then $B_d(\sigma\tilde{e}_d)\leq \mathcal{G}_{5,d}(2F)(\sigma)$. In particular, if $B_d=\mathcal{G}_{3,d}(\bar{A}_d)$, then $B_d(\sigma\tilde{e}_d)\leq \mathcal{G}_{5,d}(2\bar{A}_d)(\sigma)$. 
\end{proposition}
\begin{proof}
$B_d(\sigma\tilde{e}_d)=\mathcal{G}_{3,d}(F)(\sigma\tilde{e}_d)=\mathcal{F}_{dec,d}(2F)(\sigma\tilde{e}_d)\leq \mathcal{G}_{5,d}(2F)(\sigma)$.    
\end{proof}
Now we consider the case where the value of a function $E_d \to \mathbb{Z}_+$ only depends on $|\delta|$. 
\begin{definition}\label{4.6}
For a function $f:\mathbb{Z}_+\to\mathbb{Z}_+$, set   
$$\mathcal{H}_{1,d}(f): \sigma \mapsto \mathcal{G}_{5,d}(2f(|\cdot|))(\sigma).$$
\end{definition}
Let $f:\mathbb{Z}_+ \to \mathbb{Z}_+$ be an increasing function and $g=f(|\cdot|):E_d \to \mathbb{Z}_+$. We take $F=2g$ and this gives us the notations of $\prescript{d}{}{\psi}_{i,\sigma}(\cdot)$ and $\prescript{d}{}{\phi}_i(\cdot)$ depending on $F$. We would like to check the action of $\mathcal{G}_{5,d}$ on the function $F=2g$. In this case,
$$\prescript{d}{}{\psi_{1,\sigma}}(\epsilon)=\epsilon+2g(\sigma,\ldots,\sigma,\epsilon)=\epsilon+2f((d-1)\sigma+\epsilon);$$
$$\prescript{d}{}{\psi_{2,\sigma}}(\epsilon)=\prescript{d}{}{\psi}^\epsilon_{1,\epsilon}(\epsilon)+2g(\sigma,\ldots,\sigma,\prescript{d}{}{\psi}^\epsilon_{1,\epsilon}(\epsilon),\prescript{d}{}{\psi}^\epsilon_{1,\epsilon}(\epsilon))=\prescript{d}{}{\psi}^\epsilon_{1,\epsilon}(\epsilon)+2f((d-2)\sigma+2\cdot\prescript{d}{}{\psi}^\epsilon_{1,\epsilon}(\epsilon)).$$
And inductively, for $2 \leq i \leq d-1$,
$$\prescript{d}{}{\psi}_{i,\sigma}(\epsilon)=\prescript{d}{}{\psi}^\epsilon_{i-1,\epsilon}(\epsilon)+2g(\sigma(\tilde{e}_d-\tilde{e}_i)+\prescript{d}{}{\psi}^\epsilon_{i-1,\epsilon}(\epsilon)\tilde{e}_i)=\prescript{d}{}{\psi}^\epsilon_{i-1,\epsilon}(\epsilon)+2f((d-i)\sigma+i\cdot\prescript{d}{}{\psi}^\epsilon_{i-1,\epsilon}(\epsilon)).$$
Let $\prescript{d}{}{\phi}_i(\sigma)=\prescript{d}{}{\psi}^{\sigma}_{i,\sigma}(\sigma)$. Then
$$\mathcal{H}_{1,d}(f)(\sigma)=\mathcal{G}_{5,d}(2g)(\sigma)=\prescript{d}{}{\phi}_1\ldots\prescript{d}{}{\phi}_{d-1}(\sigma).$$
\begin{lemma}
For any increasing function $f$ and any $i,\sigma,\epsilon$, the function $\prescript{d}{}{\psi}_{i,\sigma}$ satisfies $\prescript{d}{}{\psi}_{i,\sigma}(\epsilon)\geq \epsilon$.    
\end{lemma}
\begin{proof}
By induction on $i$.    
\end{proof}
\begin{definition}
Let $\mathcal{P}_\zeta$ be a 2-package of functions $\hat{A}_d:\mathbb{Z}_+\to\mathbb{Z}_+,\hat{B}_d:\mathbb{Z}_+\to\mathbb{Z}_+, d \geq 2$. The initial condition $\mathcal{I}_\zeta$ is $\hat{A}_2(\sigma)=4\sigma+4$. The recurrence relation $\mathcal{R}_\zeta$ is given by $\hat{B}_d=\mathcal{H}_{1,d}(\hat{A}_d)$ and $\hat{A}_{d+1}=\mathcal{H}_{2,d}(\hat{B}_d)$. Here
$$\mathcal{H}_{2,d}(f):\sigma \mapsto \sigma+f(d^{2^{f(3\sigma+4)}})$$
and $\mathcal{H}_{1,d}$ is defined in Definition \ref{4.6}.
\end{definition}
\begin{theorem}
Let $\hat{A}_{\mathcal{P}_\zeta,d},\hat{B}_{\mathcal{P}_\zeta,d}$ be the functions in the package $\mathcal{P}_\zeta$. Let $\bar{A}_{\mathcal{P}_\gamma,d}, B_{\mathcal{P}_\gamma,d}$ be the functions in the package $\mathcal{P}_\gamma$. Then for any $d \geq 2$, $\sigma \in \mathbb{Z}_+$, $\delta \in E_d$, $\bar{A}_{\mathcal{P}_\gamma,d}(\delta) \leq \hat{A}_{\mathcal{P}_\zeta,d}(|\delta|)$ and $B_{\mathcal{P}_\gamma,d}(\sigma\tilde{e}_d) \leq \hat{B}_{\mathcal{P}_\zeta,d}(\sigma)$. 
\end{theorem}
\begin{proof}
We prove this by induction. By definition $\hat{A}_{\mathcal{P}_\zeta,2}(|\delta|)=4|\delta|+4=\bar{A}_{\mathcal{P}_\gamma,2}(\delta)$. Assume $\bar{A}_{\mathcal{P}_\gamma,d}(\delta) \leq \hat{A}_{\mathcal{P}_\zeta,d}(|\delta|)$ for some $d \geq 2$. The functional $\mathcal{G}_{5,d}$ is ascending, so
\begin{align*}
B_{\mathcal{P}_\gamma,d}(\sigma\tilde{e}_d)=\mathcal{G}_{3,d}(\bar{A}_{\mathcal{P}_\gamma,d})(\sigma\tilde{e}_d) \leq \mathcal{G}_{5,d}(2\bar{A}_{\mathcal{P}_\gamma,d})(\sigma)
\leq \mathcal{G}_{5,d}(2\hat{A}_{\mathcal{P}_\zeta,d}(|\cdot|))(\sigma)=\hat{B}_{\mathcal{P}_\zeta,d}(\sigma).
\end{align*}
Now assume $B_{\mathcal{P}_\gamma,d}(\sigma\tilde{e}_d) \leq \hat{B}_{\mathcal{P}_\zeta,d}(\sigma)$ for some $d \geq 2$. By the recurrence relations,
\begin{align*}
\bar{A}_{\mathcal{P}_\gamma,d+1}(\delta)=|\delta|+B_{\mathcal{P}_\gamma,d}(d^{2^{B_{\mathcal{P}_\gamma,d}((3|\delta|+4)\tilde{e}_d)}}\tilde{e}_d)\\
\leq |\delta|+B_{\mathcal{P}_\gamma,d}(d^{2^{\hat{B}_{\mathcal{P}_\zeta,d}(3|\delta|+4)}}\tilde{e}_d)\\
\leq |\delta|+\hat{B}_{\mathcal{P}_\zeta,d}(d^{2^{\hat{B}_{\mathcal{P}_\zeta,d}(3|\delta|+4)}})\\
=\mathcal{H}_{2,d}(\hat{B}_{\mathcal{P}_\zeta,d})(|\delta|)=\hat{A}_{\mathcal{P}_\zeta,d+1}(|\delta|).
\end{align*}
\end{proof}
\begin{definition}
We say an increasing function $f:\mathbb{Z}_+\to\mathbb{Z}_+$ is strictly increasing, if for $m,n \in \mathbb{Z}_+$, $f(m+n) \geq f(m)+n$. We say ascending function $f:E_d\to\mathbb{Z}_+$ is strictly ascending, if for $m \in E_d$ and $1 \leq i \leq d$, $f(m+e_i) \geq f(m)+1$.     
\end{definition}
Let $d \in \mathbb{Z}_+$, $f:\mathbb{Z}_+\to\mathbb{Z}_+$ be an strictly increasing function. We define $$\prescript{d}{}{\tilde{\psi}}_1(\epsilon)=2f((d+1)\epsilon), \prescript{d}{}{\tilde{\psi}}_i(\epsilon)=2f((d+1)\cdot\prescript{d}{}{\tilde{\psi}}^{\epsilon}_{i-1}(\epsilon)),\prescript{d}{}{\tilde{\phi_i}}(\epsilon)=\prescript{d}{}{\tilde{\psi}}^\epsilon_i(\epsilon),1\leq i\leq d-1.$$
\begin{theorem}
If $f$ is strictly increasing, then for the functions $\prescript{d}{}{\psi},\prescript{d}{}{\tilde{\psi}},\prescript{d}{}{\phi},\prescript{d}{}{\tilde{\phi}}$ defined from $F=2g=2f(|\cdot|)$ and any $\epsilon\geq \sigma$, $1 \leq i \leq d-1$,
\begin{enumerate}
\item $\prescript{d}{}{\psi}_{i,\sigma}(\epsilon)\leq \prescript{d}{}{\tilde{\psi}}_i(\epsilon)$;
\item $\prescript{d}{}{\phi}_i(\epsilon)\leq\prescript{d}{}{\tilde{\phi}}_i(\epsilon)$;
\item For any $k \geq 1$, $\prescript{d}{}{\psi}^k_{i,\epsilon}(\epsilon) \leq \prescript{d}{}{\tilde{\psi}}^k_i(\epsilon)$.
\end{enumerate}  
\end{theorem}
\begin{proof}
We prove (i) and (iii) inductively on $i$ and $k$. For $i=1$, $\prescript{d}{}{\psi}_{1,\sigma}(\epsilon)=\sigma+2f((d-1)\sigma+\epsilon)\leq 2f((d+1)\epsilon)$ since $f$ is strictly increasing and $\sigma \leq \epsilon$, so (i) holds. If (i) is true for $i-1\geq 1$, then for $i$ we have $\prescript{d}{}{\psi}_{i,\sigma}(\epsilon)=\prescript{d}{}{\psi}^\epsilon_{i-1,\epsilon}(\epsilon)+2f((d-i)\sigma+i\cdot\prescript{d}{}{\psi}^\epsilon_{i-1,\epsilon}(\epsilon))\leq 2f((d-i)\sigma+(i+1)\cdot\prescript{d}{}{\psi}^\epsilon_{i-1,\epsilon}(\epsilon))$. Since $\sigma\leq\epsilon\leq \prescript{d}{}{\psi}^\epsilon_{i-1,\epsilon}(\epsilon)$, $\prescript{d}{}{\psi}_{i,\sigma}(\epsilon)\leq 2f((d+1)\cdot\prescript{d}{}{\psi}^\epsilon_{i-1,\epsilon}(\epsilon))$. We claim that (i) holds true for $i-1$ implies (iii) holds true for $i-1$ and any $k$ by induction on $k$. First, (i) for $i-1$ just means (iii) for $i-1$ and $k=1$, so they hold at the same time. If (iii) holds for $i-1$ and $k-1$, then for $k$ we have
$$\prescript{d}{}{\psi}^k_{i-1,\epsilon}(\epsilon)=\prescript{d}{}{\psi}_{i-1,\epsilon}(\prescript{d}{}{\psi}^{k-1}_{i-1,\epsilon}(\epsilon)) \leq \prescript{d}{}{\psi}_{i-1,\epsilon}(\prescript{d}{}{\tilde{\psi}}^{k-1}_{i-1}(\epsilon)) \leq \prescript{d}{}{\tilde{\psi}}_{i-1}(\prescript{d}{}{\tilde{\psi}}^{k-1}_{i-1}(\epsilon))=\prescript{d}{}{\tilde{\psi}}^k_{i-1}(\epsilon).$$
The first inequality is true because the function $\prescript{d}{}{\psi}_{i-1,\epsilon}$ is increasing, and the second inequality is true because $\prescript{d}{}{\tilde{\psi}}^{k-1}_{i-1}(\epsilon) \geq \prescript{d}{}{\psi}^{k-1}_{i-1,\epsilon}(\epsilon)\geq\epsilon$, so we may replace $\epsilon$ by $\prescript{d}{}{\tilde{\psi}}^{k-1}_{i-1}(\epsilon)$ and $\sigma$ by $\epsilon$, so by induction hypothesis the inequality holds. So (iii) is true for $i-1$ and any $k \geq 1$. In particular, (iii) is true for $i-1$ and $k=\epsilon$, so we have $\prescript{d}{}{\psi}_{i,\sigma}(\epsilon)\leq 2f((d+1)\cdot\prescript{d}{}{\psi}^\epsilon_{i-1,\epsilon}(\epsilon))\leq 2f((d+1)\cdot\prescript{d}{}{\tilde{\psi}}^\epsilon_{i-1}(\epsilon))=\prescript{d}{}{\tilde{\psi}}_i(\epsilon)$, so (i) is true for $i$, so we get (i) and (iii) by induction. (ii) is the particular case of (iii) when $k=\epsilon$.
\end{proof}
\begin{definition}\label{4.12}
For $d \geq 2$, we define the following functional $\mathcal{H}_{3,d}$ from a function $A_d:\mathbb{Z}_{\geq 2} \to \mathbb{Z}_{\geq 2}$ to $B_d:\mathbb{Z}_{\geq 2} \to \mathbb{Z}_{\geq 2}$ by the following equation:
$$\prescript{d}{}{\psi}_1(\epsilon)=2A_d((d+1)\epsilon),\prescript{d}{}{\psi}_i(\epsilon)=2A_d((d+1)\psi^\epsilon_{i-1}(\epsilon)),2\leq i \leq d-1,$$
$$\prescript{d}{}{\phi}_i(\epsilon)=\prescript{d}{}{\psi}^{\epsilon}_i(\epsilon), B_d(\sigma)=\prescript{d}{}{\phi}_1\circ\prescript{d}{}{\phi}_2\circ\ldots\circ\prescript{d}{}{\phi}_{d-1}(\sigma).$$
If the above equations are satisfied, we say $B_d=\mathcal{H}_{3,d}(A_d)$.
\end{definition}
\begin{proposition}\label{4.13}
If $f$ is strictly increasing, then we have
$$\mathcal{H}_{1,d}(f) \leq \mathcal{H}_{3,d}(f).$$
\end{proposition}
\begin{notations}
For $d \geq 2$, we define a package $\mathcal{P}_\theta$ of two sequences of functions $A_d,B_d:\mathbb{Z}_{\geq 2}\to \mathbb{Z}_{\geq 2}$ and we define two sets of functions associated to this package indexed by integers $1 \leq i\leq d-1$: $\prescript{d}{}{\psi}_i,\prescript{d}{}{\phi}_i:\mathbb{Z}_+\to\mathbb{Z}_+$. They are functions defined in the following way:
\begin{enumerate}
\item Let $A_2(\sigma)=4\sigma+4$. 
\item For function $f$, $\mathcal{H}_{2,d}(f)\colon \sigma \mapsto \sigma+f(d^{2^{f(3\sigma+4)}})$, then $A_{d+1}=\mathcal{H}_{2,d}(B_d)$.
\item $B_d=\mathcal{H}_{3,d}(A_d)$.
\end{enumerate}
Inductively we can prove if $\sigma \geq 2$, then $A_d(\sigma),B_d(\sigma) \geq 2$, so they are well-defined.
\end{notations}
\begin{theorem}
For any $d \geq 2$, all the functions in the package $\mathcal{P}_\zeta$ and $\mathcal{P}_\theta$ are strictly increasing, and $A_{\mathcal{P}_\theta,d} \geq \hat{A}_{\mathcal{P}_\zeta,d}$ and $B_{\mathcal{P}_\theta,d} \geq \hat{B}_{\mathcal{P}_\zeta,d}$ on $\mathbb{Z}_{\geq 2}$.   
\end{theorem}
\begin{proof}
The strictly increasing property can be proved inductively since the common initial condition $A_2$ is strictly increasing, and both recurrence relations map strictly increasing functions to strictly increasing functions. Thus the inequality can be derived from Proposition \ref{4.13}.    
\end{proof}
\section{Upper bound of $B_{\mathcal{P}_\theta,d}$ in terms of Knuth arrows}
In this section, we find an upper bound of $B_{\mathcal{P}_\theta,d}$ in terms of Knuth arrows. First we recall the definition of Knuth arrows:
\begin{definition}[Knuth arrows]
Assume $a,b,k \in \mathbb{Z}_+$. We denote
\begin{align*}
a\uparrow b=a^b\\
a\uparrow^k b=\underbrace{a\uparrow^{k-1}\ldots\uparrow^{k-1} a}_{a \textup{ appears } b \textup{ times}}.
\end{align*}   
\end{definition}
\begin{example}
We have:
\begin{enumerate}
\item $4\uparrow4=4^4=256.$
\item $4\uparrow\uparrow4=4\uparrow4\uparrow4\uparrow4=4^{4^{4^4}}=4^{4^{256}},$
\item $4\uparrow^34=4\uparrow\uparrow4\uparrow\uparrow4\uparrow\uparrow4=4\uparrow\uparrow4\uparrow\uparrow(4^{4^{256}})=4\uparrow\uparrow(4\uparrow\ldots\uparrow4)=4\uparrow\uparrow(4^{\ldots^4}).$
where the height of the small tower is $4^{4^{256}}$. The value of $4\uparrow\uparrow(4^{\ldots^4})$ is a larger power tower of $4$ whose height is the value of the small tower.
\end{enumerate}
\end{example}
When we evaluate an expression in terms of multiple Knuth arrows without parentheses, we always compute from right to left. For example,
$$4\uparrow\uparrow2\uparrow2\uparrow^53=4\uparrow\uparrow(2\uparrow(2\uparrow^53)).$$
Knuth arrow can express large numbers and its growth in terms of the inputs is extremely rapid. Here, we introduce some properties of Knuth arrows.
\begin{lemma}
We assume $a,b \geq 2,c,k \geq 1$.
\begin{enumerate}
\item The function $(a,k,b) \to a\uparrow^k b$ is strictly increasing in each variable.
\item $a\uparrow^k b \geq a^b \geq ab \geq a+b$. The first inequality becomes an equality if and only if $k=1$ or $a=b=2$, and the second and the third inequalities become equalities if and only if $a=b=2$.
\item $c+a\uparrow^k b \leq a\uparrow^k(b+c),c\cdot (a\uparrow^k b) \leq a\uparrow^k(cb)$.
\item (Left absorbing property) $a\uparrow^k a\uparrow^k\ldots a\uparrow^k(a\uparrow^{k+1}b)=a\uparrow^{k+1}(b+c)$, where $c$ is the number of the symbol $\uparrow^k$.
\end{enumerate}
\end{lemma}
The following estimation of the value of iterated Knuth arrow is essential in our inductive proof. Although this estimation is very rough, it is extremely convenient to use.
\begin{lemma}
Let $k_1,k_2,\ldots,c_1,c_2,\ldots,d_1,d_2,\ldots$ be integers at least one and $a_1,a_2,\ldots$ be integers at least two. Consider the expression
$$c_1\cdot(a_1\uparrow^{k_1}(c_2\cdot(a_2\uparrow^{k_2}(\ldots(\sigma)\ldots))+d_2))+d_1.$$
Let $k$ be the maximum of $2$ and all $k_i$'s, $a$ be the maximum of all $a_i,c_i,d_i$'s. Let $n$ be the number of all $i:c_i \neq 1$ plus the number of all $a_i$ and nonzero $d_i$'s. Then the above expression is no greater than
$$a\uparrow^{k+1}(\sigma+n-1).$$
\end{lemma}
\begin{proof}
We make the following observation:
\begin{enumerate}
\item If $c=1$ and $d \geq 1$, then $c\sigma+d \leq \max\{2,d\}\uparrow\sigma$;
\item If $c \geq 2, d \geq 1$, then $c\sigma+d \leq \max\{2,d\}\uparrow c\uparrow \sigma.$
\end{enumerate}
Thus we can convert $+d$ into $d\uparrow$ and $c\cdot$ into $c\uparrow$. We get a chain of Knuth arrows such that each arrow is at most $\uparrow^k$, each entry is at most $a \geq 2$, and the number of the terms before $\sigma$ is just $n$. Thus the above expression is at most the following chain with $n$ many $a$'s:
$$a\uparrow^k a \uparrow^k \ldots a \uparrow^k\sigma,$$
which is smaller than
$$a\uparrow^k a \uparrow^k \ldots a \uparrow^{k+1}\sigma=a\uparrow^{k+1}(\sigma+n-1).$$
\end{proof}
For example,
$$4\uparrow\uparrow2\uparrow2\uparrow^53\leq \max\{2,4\}\uparrow^{(5+1)}(3+2)=4\uparrow^65.$$

\begin{notations}
We mark the functions in Definition \ref{4.12} by their index, which is the order they appear in the induction process. We arrange the functions in the following order:
$$A_2,\prescript{2}{}{\psi}_1,\prescript{2}{}{\phi}_1,B_2;$$
$$A_3,\prescript{3}{}{\psi}_1,\prescript{3}{}{\phi}_1,\prescript{3}{}{\psi}_2,\prescript{3}{}{\phi}_2,B_3;$$
$$A_4,\prescript{4}{}{\psi}_1,\prescript{4}{}{\phi}_1,\prescript{4}{}{\psi}_2,\prescript{4}{}{\phi}_2,\prescript{4}{}{\psi}_3,\prescript{4}{}{\phi}_3,B_4.$$
For a function $f$ appearing in this sequence, the index of the function is denoted by $\ind(f)$. Also for integers $d \geq 2$, denote $\tilde{d}=d+1 \geq 3$.
\end{notations}
\begin{theorem}
Let $\sigma \geq 2$. Under the notations in Definition \ref{4.12}, for any function $f$ in the sequence, we have
$$f(\sigma) \leq \tilde{d}\uparrow^{\ind(f)}\tilde{d}\uparrow\sigma.$$
In particular, $B_d(\sigma)\leq \tilde{d}\uparrow^{d^2+d-2}\tilde{d}\uparrow\sigma$.
\end{theorem}
\begin{proof}
We prove by induction. We will apply the inequalities for $\sigma \geq 2$, which are easy to verify: $\sigma+5 \leq \sigma+7 \leq 3\uparrow\sigma$, $3\sigma+4 \leq 4\sigma+4\leq 3\uparrow3\uparrow\sigma$.

The base case: $A_2(\sigma)=4\sigma+4\leq 3\uparrow3\uparrow\sigma.$

The induction steps from $A_d$ to $\prescript{d}{}{\psi}_1(\sigma)$: we have $\prescript{d}{}{\psi}_1(\sigma)=2A_d((d+1)
\sigma)$. So
$$\prescript{d}{}{\psi}_1(\sigma)=2A_d((d+1)
\sigma)\leq 2\uparrow \tilde{d}\uparrow^{\ind(A_d)}\tilde{d}\uparrow\tilde{d}\sigma\leq \tilde{d}\uparrow^{\ind(A_d)+1}(\sigma+3)$$
$$\leq \tilde{d}\uparrow^{\ind(A_d)+1}3\uparrow\sigma\leq \tilde{d}\uparrow^{\ind(\prescript{d}{}{\psi}_1)}\tilde{d}\uparrow\sigma.$$

The induction steps from $\prescript{d}{}{\psi}_i(\sigma)$ to $\prescript{d}{}{\phi}_i(\sigma)=\prescript{d}{}{\psi}^{\sigma}_i(\sigma)$: since $\prescript{d}{}{\psi}_i(\sigma)\leq\tilde{d}\uparrow^{\ind(\prescript{d}{}{\psi}_i)}\tilde{d}\uparrow\sigma$, we have
$$\prescript{d}{}{\phi}_i(\sigma)\leq\tilde{d}\uparrow^{\ind(\prescript{d}{}{\psi}_i)}\tilde{d}\uparrow\ldots\tilde{d}\uparrow^{\ind(\prescript{d}{}{\psi}_i)}\tilde{d}\uparrow\sigma,$$
where the symbol $\tilde{d}\uparrow^{\ind(\prescript{d}{}{\psi}_i)}\tilde{d}\uparrow$ is repeated $\sigma$ times. Therefore,
$$\prescript{d}{}{\phi}_i(\sigma)\leq\tilde{d}\uparrow^{\ind(\prescript{d}{}{\psi}_i)+1}(3\sigma)\leq\tilde{d}\uparrow^{\ind(\prescript{d}{}{\psi}_i)+1}3\uparrow\sigma\leq\tilde{d}\uparrow^{\ind(\prescript{d}{}{\phi}_i)}\tilde{d}\uparrow\sigma.$$

The induction steps from $\prescript{d}{}{\phi}_i(\sigma)$ to $\prescript{d}{}{\psi}_{i+1}(\sigma)$:
$$\prescript{d}{}{\psi}_{i+1}(\sigma)=2A_d((d+1)\prescript{d}{}{\phi}_i(\sigma))\leq 2\cdot \tilde{d}\uparrow^{\ind(A_d)}\tilde{d}\uparrow((d+1)\tilde{d}\uparrow^{\ind(\prescript{d}{}{\phi}_i)}\tilde{d}\uparrow\sigma).$$
We have $\ind(A_d)\leq \ind(\prescript{d}{}{\phi}_i)$. There are $6$ terms on the right side of the equation with five $\tilde{d}=d+1$'s, and one $2$. Therefore, 
$$\prescript{d}{}{\psi}_{i+1} \leq \tilde{d}\uparrow^{\ind(\prescript{d}{}{\phi}_i)+1}(\sigma+5)\leq \tilde{d}\uparrow^{\ind(\prescript{d}{}{\phi}_i)+1}3\uparrow\sigma\leq\tilde{d}\uparrow^{\ind(\prescript{d}{}{\psi}_{i+1})}\tilde{d}\uparrow\sigma.$$

The induction steps from $\prescript{d}{}{\phi}_i(\sigma),1\leq i\leq d-1$ to $B_d(\sigma)=\prescript{d}{}{\phi}_1\prescript{d}{}{\phi}_2\ldots\prescript{d}{}{\phi}_{d-1}(\sigma)$:
$$B_d(\sigma)\leq \tilde{d}\uparrow^{\ind(\prescript{d}{}{\phi}_1)}\tilde{d}\uparrow\tilde{d}\uparrow^{\ind(\prescript{d}{}{\phi}_2)}\tilde{d}\uparrow\ldots\tilde{d}\uparrow^{\ind(\prescript{d}{}{\phi}_{d-1})}\tilde{d}\uparrow\sigma,$$
where the symbol $\tilde{d}\uparrow^{\ind(\prescript{d}{}{\phi}_i)}\tilde{d}\uparrow$ is repeated by $d-1$ times, and for any $1 \leq i \leq d-1$, $\ind(\prescript{d}{}{\phi}_i)\leq \ind(\prescript{d}{}{\phi}_{d-1})$. Therefore,
$$B_d(\sigma)\leq \tilde{d}\uparrow^{\ind(\prescript{d}{}{\phi}_{d-1})+1}(2(d-1)+\sigma). 
$$
From $\sigma \geq 2$ we get $(\sigma-2)(d-1) \geq 0$, so $2(d-1)+\sigma \leq \sigma d$, so
$$
B_d(\sigma)\leq\tilde{d}\uparrow^{\ind(\prescript{d}{}{\phi}_{d-1})+1}(d\sigma) \leq\tilde{d}\uparrow^{\ind(\prescript{d}{}{\phi}_{d-1})+1}d\uparrow\sigma\leq\tilde{d}\uparrow^{\ind(B_d)}\tilde{d}\uparrow\sigma.$$

The induction steps from $B_{d-1}(\sigma)$ to $A_d(\sigma)=\sigma+B_{d-1}((d-1)\uparrow2\uparrow B_{d-1}(3\sigma+4))$:
\begin{align*}
A_d(\sigma)\leq \sigma+\tilde{d}\uparrow^{\ind(B_{d-1})}\tilde{d}\uparrow (d-1)\uparrow2\uparrow\tilde{d}\uparrow^{\ind(B_{d-1})}\tilde{d}\uparrow(3\sigma+4)\\
\leq \tilde{d}\uparrow^{\ind(B_{d-1})}\tilde{d}\uparrow (d-1)\uparrow2\uparrow\tilde{d}\uparrow^{\ind(B_{d-1})}\tilde{d}\uparrow(4\sigma+4)\\
\leq \tilde{d}\uparrow^{\ind(B_{d-1})}\tilde{d}\uparrow (d-1)\uparrow2\uparrow\tilde{d}\uparrow^{\ind(B_{d-1})}\tilde{d}\uparrow3\uparrow3\uparrow\sigma.
\end{align*}
Here we use $\widetilde{d-1}\leq \tilde{d}$. We have $2,3,d-1\leq \tilde{d}$, and there are $8$ terms except for $\sigma$; they are $\tilde{d},\tilde{d},d-1,2,\tilde{d},\tilde{d},3,3$. So
$$A_d(\sigma)\leq \tilde{d}\uparrow^{\ind(B_{d-1})+1}(\sigma+7)\leq\tilde{d}\uparrow^{\ind(B_{d-1})+1}(3\uparrow\sigma)\leq  \tilde{d}\uparrow^{\ind(A_d)}\tilde{d}\uparrow\sigma.$$
So the first claim is proved. The index of $B_d$ is $4+6+\ldots+2d=2(d(d+1)/2)-2=d^2+d-2$, so by the first claim, $B_d(\sigma)\leq \tilde{d}\uparrow^{d^2+d-2}\tilde{d}\uparrow\sigma$.
\end{proof}
Using the same proof in the induction process from $A_d$ to $B_d$, we get:
\begin{theorem}\label{5.7}
Let $f:\mathbb{Z}_{\geq 2} \to \mathbb{Z}_{\geq 2}$, $d_1 \in \mathbb{Z}$, $d_1 \geq \tilde{d}$. If $f(\sigma)\leq d_1\uparrow^k\ d_1\uparrow\sigma$, then $\mathcal{H}_{3,d}(f) \leq d_1\uparrow^{k+\ind(B_d)-\ind(A_d)}d_1\uparrow\sigma=d_1\uparrow^{k+2d-1}d_1\uparrow\sigma$.    
\end{theorem}
Now we gather all the inequalities above to get
$$B_{\mathcal{P}_\alpha,d}(3,\sigma\tilde{e}_d)\leq B_{\mathcal{P}_\beta,d}(3,\sigma\tilde{e}_d) \leq B_{\mathcal{P}_\gamma,d}(3,\sigma\tilde{e}_d) \leq \hat{B}_{\mathcal{P}_\zeta,d}(\sigma)\leq B_{\mathcal{P}_\theta,d}(\sigma) \leq \tilde{d}\uparrow^{d^2+d-2}\tilde{d}\uparrow\sigma.$$
Therefore, we have:
\begin{theorem}
Let $d \in \mathbb{Z}_{\geq 2}$, $\tilde{d}=d+1$, $\delta \in E_d$ be a dimension sequence. Assume $\sigma=\max\{2,\delta_i:1\leq i \leq d\}$. Then for any $V$ with $\underline{\dim}V=\delta$, $V$ lies inside a $K$-algebra generated by an $R_3$-sequence of forms of degree at most $d$ and length at most $\tilde{d}\uparrow^{d^2+d-2}\tilde{d}\uparrow\sigma$.    
\end{theorem}
The bound on the projective dimension is an immediate corollary of the above theorem.
\begin{theorem}\label{5.9}
Let $d \in \mathbb{Z}_{\geq 2}$, $\tilde{d}=d+1$, $\delta \in E_d$ be a dimension sequence. Assume $\sigma=\max\{2,\delta_i:1\leq i \leq d\}$, and $V$ is a graded vector space with $\underline{\dim}V=\delta$. Let $I$ be the $R$-ideal generated by $V$, then $\pd(R/I) \leq \tilde{d}\uparrow^{d^2+d-2}\tilde{d}\uparrow\sigma$.    
\end{theorem}
\begin{proof}
Let $B=\tilde{d}\uparrow^{d^2+d-2}\tilde{d}\uparrow\sigma$ and $x_1,\ldots,x_B$ be the $R_3$-sequence such that $V \subset K[x_1,\ldots,x_B]=S$. Since $x_1,\ldots,x_B$ is a regular sequence, $S$ is isomorphic to a graded polynomial ring, and the map $S \to R$ is flat and graded. Now by flatness we have $B \geq \pd(S/VS)=\pd(R/VS\cdot R)=\pd(R/I)$.   
\end{proof}

\section{Other related bounds}
In this section, we will use the bound in Section 5 to derive bounds on functions satisfying other properties. Denote $\tilde{d}=d+1$.
\begin{theorem}\label{6.1}
Assume $d \geq 2$ and $D_d(\sigma)=d\uparrow2\uparrow \tilde{d}\uparrow^{d^2+d-2}\tilde{d}\uparrow\sigma$. Then $D_d \geq D_{\mathcal{P}_\alpha,d}(\sigma)$, so $D_d$ satisfies $PD_d$. That is, for any $R$-ideal $I$ generated by a regular sequence of length at most $\sigma$ and degree at most $d$, assume $P$ is a minimal prime over $I$, then $P$ is generated by at most $D_d(\sigma)$ elements.   
\end{theorem}
\begin{proof}
We have
\begin{align*}
D_{\mathcal{P}_\alpha,d}(\sigma)=d\uparrow2\uparrow (B_{\mathcal{P}_\alpha,d}(1,\sigma\tilde{e}_d))\leq d\uparrow2\uparrow B_{\mathcal{P}_\alpha,d}(3,\sigma\tilde{e}_d)\leq d\uparrow2\uparrow \tilde{d}\uparrow^{d^2+d-2}\tilde{d}\uparrow\sigma.
\end{align*}
\end{proof}
\begin{theorem}\label{6.2}
Assume $\eta \geq 2$. We set $A_2(\eta)=\eta+2$, and for $d \geq 3$, set
$$A_d(\eta)=d\uparrow^{d^2-d-1}3\uparrow\eta.$$
We also define $A_d(1)=A_d(2)$ for any $d\geq 2$. Then $A_d(\eta) \geq A_{\mathcal{P}_\alpha,d}(\eta)$, so $A_d(\eta)$ satisfies $PAa_d$. That is, for any $F \in R_d$ such that $\str(F) \geq A_d(\eta)$, the ring $R/FR$ satisfies Serre's condition $R_\eta$.
\end{theorem}
\begin{proof}
For $d=2$, $A_2=A_{\mathcal{P}_\alpha,2}$. Now assume $d \geq 3$, then
\begin{align*}
A_{\mathcal{P}_\alpha,d}(\eta)=\mathcal{F}_{1,d}(\mathcal{B}_{\mathcal{P}_\alpha,d-1},D_{\mathcal{P}_\alpha,d-1})(\eta)\\
=\mathcal{B}_{\mathcal{P}_\alpha,d-1}(3,D_{\mathcal{P}_\alpha,d-1}(\eta+1))+1=B_{\mathcal{P}_\alpha,d-1}(3,D_{\mathcal{P}_\alpha,d-1}(\eta+1)\tilde{e}_d)+1\\
\leq \widetilde{d-1}\uparrow^{d^2-d-2}\widetilde{d-1}\uparrow D_{\mathcal{P}_\alpha,d-1}(\eta+1)+1\\
\leq \widetilde{d-1}\uparrow^{d^2-d-2}\widetilde{d-1}\uparrow(d-1)\uparrow2\uparrow\widetilde{d-1}\uparrow^{d^2-d-2}\widetilde{d-1}\uparrow(\eta+1)+1.
\end{align*}
All the terms except for $\eta$ are no larger than $d$ and there are $8$ such terms, so the above expression is smaller than
$$d\uparrow^{d^2-d-1}(\eta+7)\leq d\uparrow^{d^2-d-1}3\uparrow\eta=A_d(\eta).$$
\end{proof}
\begin{theorem}\label{6.3}
We set $\bar{A}_2(\eta,\delta)=\eta+4|\delta|+1$, and for $d \geq 3$, set $\bar{A}_d(\eta,\delta)=|\delta|-1+d\uparrow^{d^2-d-1}3\uparrow(\eta+3|\delta|)$. Then $\bar{A}_d(\eta,\delta) \geq \bar{A}_{\mathcal{P}_\alpha,d}(\eta,\delta)$, so $\bar{A}_d(\eta,\delta)$ satisfies $PAb_d$. That is, for any graded vector space $V$ with $\underline{\dim}V=\delta$, $\str(V_j)\geq\bar{A}_d(\eta,\delta)$ for $1 \leq j \leq d$ implies that any homogeneous $K$-basis of $V$ is an $R_\eta$-sequence. 
\end{theorem}
\begin{proof}
Apply the ascending functor $\mathcal{F}_{2,d}$ to the inequality $A_d \geq A_{\mathcal{P}_\alpha,d}$.   
\end{proof}
\begin{theorem}\label{6.4}
Assume $d \geq 2, \eta \geq 4$. We set $d_1=\max\{d+1,\eta-1\}$, $B_d(\eta,\delta)=d_1\uparrow^{d^2+d-1}d_1\uparrow|\delta|$, then $B_d$ satisfies $PB_d$, that is, for any graded vector space $V$ with $\underline{\dim}V=\delta$, $V$ lies in a $K$-algebra generated by an $R_\eta$-sequence of degree at most $d$ and length at most $B_d(\eta,\delta)$.  
\end{theorem}
\begin{proof}
If $|\delta|=1$, then it is trivial. We assume $|\delta|\geq 2$, so $4|\delta|\leq 3\uparrow|\delta|$. We see for any $d$,
\begin{align*}
\bar{A}_d(\eta,\delta)\leq d\uparrow^{d^2-d-1}3 \uparrow(\eta+4|\delta|-1)\leq d\uparrow^{d^2-d-1}d_1 \uparrow(\eta+4|\delta|-1)\\
\leq d_1\uparrow^{d^2-d-1}d_1 \uparrow(\eta-1)\uparrow 3\uparrow|\delta|\leq d_1\uparrow^{d^2-d}(|\delta|+3)\leq d_1\uparrow^{d^2-d}d_1\uparrow|\delta|.
\end{align*}
By Lemma \ref{3.11}, if $\bar{A}_d$ satisfies $PAb_d$, then $(\eta,\delta) \to \mathcal{H}_{3,d}(\bar{A}_d(\eta,\cdot))(\delta)$ satisfies $PB_d$. By Theorem \ref{5.7} we have
$$\mathcal{H}_{3,d}(\bar{A}_d(\eta,\cdot))(\delta) \leq d_1\uparrow^{d^2-d+2d-1}d_1\uparrow|\delta|=d_1\uparrow^{d^2+d-1}d_1\uparrow|\delta|.$$
\end{proof}
\begin{theorem}\label{6.5}
Assume $d \geq 2$. Let $M$ be an $R$-module with minimal presentation
$$R^n \xrightarrow[]{\phi} R^m \to M \to 0,$$
where entries of $\phi$ are elements in $R$ which are not necessarily homogeneous and have degree at most $d$. Then $\pd_R(M) \leq \tilde{d}\uparrow^{d^2+d-2}\tilde{d}\uparrow(mn)$.
\end{theorem}
\begin{proof}
Consider the set of homogeneous components of entries of $\phi$. Let $V$ be the graded vector space generated by these elements.  Then in each degree, there are at most $mn$ linearly independent elements, so $\underline{\dim}V \leq mn\tilde{e}_d$. If $m=n=1$, there is nothing to prove; $M=R/fR$ must have projective dimension $1$. Otherwise $mn \geq 2$, so in this case $V$ lies in a $K$-algebra $S$ generated by $R_3$-sequence of forms of length at most $\tilde{d}\uparrow^{d^2+d-2}\tilde{d}\uparrow(mn)$, and $S$ is isomorphic to a graded polynomial ring. In particular, entries of $\phi$ lie in $S$, so there exists a graded $S$-module $N$ with $M=N\otimes_S R$. Since $S \to R$ is a graded flat map, $\pd_R M=\pd_S N \leq \tilde{d}\uparrow^{d^2+d-2}\tilde{d}\uparrow(mn)$.  
\end{proof}
\begin{theorem}\label{6.6}
For $d \geq 2$, there exists function $\Phi_d(h)$ such that $F \in R_d$ with $\str(F) \geq \Phi_d(h)$, the ideal generated by $\{\frac{\partial F}{\partial x_i}\}_{1 \leq i \leq N}$ is not contained in an ideal generated by $h$ forms of degree at most $d-1$. We can take $\Phi_d(h)=h+1$ for $d=2$ and $\Phi_d(h)=d\uparrow^{d^2-d-2}d\uparrow\max\{2,h\}+1$ for $d \geq 3$.   
\end{theorem}
\begin{proof}
By Proposition 2.6 of \cite{AH1}, if a function $\tilde{B}_{d-1}(\eta,\cdot)$ satisfies $PcB_{d-1}$, then $\Phi_d(h)=\mathcal{B}_{d-1}(3,h)+1$ satisfies the property in the statement. When $d=2$, we can take $\mathcal{B}_1(\eta,h)=h$, so $\Phi_d(h)=h+1$ would work. If $d \geq 3$, we see $\mathcal{B}_{d-1}(3,h)=\widetilde{d-1}\uparrow^{d^2-d-2}\widetilde{d-1}\uparrow\max\{2,h\}=d\uparrow^{d^2-d-2}d\uparrow\max\{2,h\}$ satisfies $PcB_{d-1}$ for $\eta=3$, so we can take $\Phi_d(h)=d\uparrow^{d^2-d-2}d\uparrow\max\{2,h\}+1$.   
\end{proof}

\section{Explicit Eisenbud-Goto bound}
In this section, our focus is on the Eisenbud-Goto conjecture:
\begin{conjecture}
Let $K$ be an algebraically closed field. If $P \subset (x_1,\ldots,x_N)^2$ is a homogeneous prime ideal, then
$$\reg(P) \leq e(R/P)-\het(P)+1.$$
\end{conjecture}

This is a conjecture on the upper bound of regularity of $R/P$ in terms of its multiplicity and height, where $P$ is a prime ideal in $R$. In this section we assume $h=\het(P)$ and $e=e(R/P)$ are known and give an explicit bound of $\reg(P)$ in terms of $h,e$ using Knuth arrows. 

When $K$ is algebraically closed, the following well-known result gives an upper bound of $h$ in terms of $e$: 
\begin{lemma}[\cite{hleqe}]\label{7.3}
Assume $K$ is algebraically closed and $P$ is nondegenerate, that is, $P_1=0$. Then $h \leq e-1$.    
\end{lemma}
Next, we derive a lemma which shows that after a general linear coordinate change, a nondegenerate prime $P$ can be written as a colon ideal of elements of bounded degree. In the following lemma, we will call $\min\{i \in \mathbb{N}:P_i\neq0\}$ the initial degree of $P$.
\begin{lemma}\label{7.2}
Assume $|K|=\infty$. After a general linear transformation, $P$ contains a regular sequence $f_1,\ldots,f_h$ of degree $a_i$ such that $in_{<_{lex}}(f_i)=x_i^{a_i}$ and $a_i \leq e$ for any $i$, and $a_1$ is the initial degree of $P$. As a consequence, we can find $f_1,\ldots,f_h$ of degree $e$ such that $P=(f_1,\ldots,f_h):f$, where $\deg f_i=e$, $\deg f\leq h(e-1)$. 
\end{lemma}
\begin{proof}
After a general linear transformation, we may assume $P$ is in general coordinates, and $in_{<_{lex}}(P)=gin_{<_{lex}}(P)$. In particular, $in_{<_{lex}}(P)$ is Borel-fixed. The initial degree and multiplicity are invariant under a general linear transformation, so the inequality we want to prove does not change. We prove the result for all $P$ such that $in_{<_{lex}}(P)$ is Borel-fixed by induction on $h=\het(P)$.

If $h=1$, then $P=(f)$ with $\deg(f)=e$ is the initial degree of $P$. In general coordinates we must have $in_{<_{lex}}f=x_1^e$, so we are done in this case.

Now assume $h\geq 2$ and assume the result holds for primes of height at most $h-1$. Let $P \subset R$ be a homogeneous prime of height $h$. Consider the weight vector $w=(1,0,\ldots,0)$, then two monomials $u,v$ satisfying $u<_w v$ must satisfy $u<_{lex} v$, that is, the total order $<_{lex}$ refines the relation $<_w$. Therefore, $in_{<_{lex}}(P)=in_{<_{lex}}(in_{<_w}(P))$. Set $Q=in_{<_w}(P)$. Let $a$ be the initial degree of $P$. Since $P$ is in general coordinates, there is $f_1 \in P$ such that $x_1^a=in_{<_{lex}}(f_1) \in in_{<_{lex}}(P)$. This implies $x_1^a \in in_{<_w}(P)$. So we can write
$$Q=(x_1^a)+(x_1^{a-1})Q_1+\ldots+x_1Q_{a-1}+Q_a,$$
where $Q_i \subset \bar{R}=K[x_2,\ldots,x_N]$. Moreover, $Q_a=P\cap \bar{R}$ is a nondegenerate homogeneous prime in $\bar{R}$. We have $e=e(P)=e(Q)$ since taking the initial ideal of a homogeneous ideal with respect to a weight does not change the multiplicity. We also have $Q \subset (x_1)+Q_a \subset \sqrt{Q}$, but $(x_1)+Q_a$ is prime, hence radical, so $(x_1)+Q_a=\sqrt{Q}$. Therefore
$$e(\bar{R}/Q_a)=e(R/((x_1)+Q_a))=e(R/\sqrt{Q})\leq e(R/Q)=e.$$
Also $x_1$ is a nonzero divisor on $R/Q_a$, so $\het(Q)=\het(Q_a)+1$. Also, $in_{<_{lex}}(Q_a)=in_{<_{lex}}(P)\cap \bar{R}$ is also Borel-fixed, so $Q_a$ satisfies the induction hypothesis. So by the induction hypothesis for $h-1$, we can pick a regular sequence $f_2,\ldots,f_h \in P$ of degree $a_i,i \geq 2$ such that $in_{<_{lex}}(f_i)=x_i^{a_i} \in in_{<_{lex}}(Q_a)\subset in_{<_{lex}}(P)$ and $a_i \leq e(\bar{R}/Q_a)\leq e$ for any $i$. So $f_1,\ldots,f_h$ is the desired sequence for $P$ in the first half of the statement. For the second half of the statement, $x_1^{e-a_1}f_1,\ldots,x_h^{e-a_h}f_h$ is the desired sequence for $P$, which is a regular sequence since passing to the initial terms they form a regular sequence. The existence of $f$ comes from Lemma \ref{colon}.
\end{proof}

For such $f_1,\ldots,f_h,f$ in Lemma \ref{7.2}, we have a short exact sequence
$$0 \to R/P \to R/(f_1,\ldots,f_h) \to R/(f_1,\ldots,f_h,f) \to 0.$$
If we can find a regular sequence $y_1,\ldots,y_B$ of homogeneous elements such that $f_1,\ldots,f_h,f \in S=K[y_1,\ldots,y_B]$, then we have another short exact sequence
$$0 \to S/Q \to S/(f_1,\ldots,f_h)S \to S/(f_1,\ldots,f_h,f)S \to 0$$
and $S \to R$ is faithfully flat. So the first short exact sequence is just the second short exact sequence base change from $S$ to $R$. In particular, $P=QR$ and $R/P=R \otimes_S S/Q$.

The subring $S$ may not be standard graded. In fact, $S$ is graded if we set $d_i=\deg y_i$, and the ideals $(f_1,\ldots,f_h,f)$ and $Q$ will be homogeneous with respect to this grading. We consider another flat extension $T=K[z_1,\ldots,z_B]$ of $S$; $S \to T$ with $y_i \to z_i^{d_i}$. We assume $S$ is standard graded, then $S \to T$ is a graded homomorphism. Thus an element is homogeneous in $S$ if and only if its image in $R$ is homogeneous, if and only if its image in $T$ is homogeneous, and the two images have the same degree. Therefore, $QT$ is a homogeneous ideal in $T$; if $P=QR$ is generated by a graded vector space with dimension vector $\delta$, then so is $QT$ in $T$. Let
$$F_\bullet \to S/Q$$
be a graded minimal $S$-free resolution of $S/Q$. Then $F_\bullet \otimes_S R$ is a graded $R$-free resolution of $R/P$, which is minimal; $F_\bullet \otimes_S T$ is a graded $T$-free resolution of $T/QT$ which is minimal. The entries of $F_\bullet \otimes_S R$ and $F_\bullet \otimes_S T$ have the same degrees. In particular, this tells us that
$\reg(R/P)=\reg(T/QT)$ and $\pd(R/P)=\pd(T/QT)$. Note that we have a third short exact sequence
$$0 \to T/QT \to T/(f_1,\ldots,f_h)T \to T/(f_1,\ldots,f_h,f)T \to 0.$$
Here we are abusing the notation; we use the fact that $f_i,f \in S$ and denote their images in $T$ under the ring homomorphism $S \to T$ still by $f_i,f$. This leads to an upper bound of $\reg(R/P)$ and $\pd(R/P)$. We need a lemma by Caviglia and Sbarra:
\begin{lemma}\label{CS1}
Let $I$ be a homogeneous $K[x_1,\ldots,x_n]$-ideal generated in degree less than or equal to $d$. Then
$$\reg(I) \leq (2d)^{2^{n-2}}.$$
\end{lemma}
Note that if $h=1$ or $e=1$, this is a trivial case. So we may assume $h,e \geq 2$.
\begin{theorem}\label{7.5} 
Let $P$ be a prime in $R$ of height $h \geq 2$ with $\deg(R/P)=e \geq 2$. Then
\begin{align*}
\reg(R/P) \leq (eh)\uparrow^{e^2h^2-1}(h+3).
\end{align*}
If moreover $K$ is algebraically closed, then
\begin{align*}
\reg(R/P) \leq (e^2)\uparrow^{e^4-1}(e+2).
\end{align*}
\end{theorem}
\begin{proof}
Choose $f_1,\ldots,f_h,f$ as in Lemma \ref{7.2}. $f_1,\ldots,f_h$ have degree $e$ and $f$ has degree $h(e-1)$. By Theorem \ref{5.9}, let $d=h(e-1)$, $\tilde{d}=d+1$, $\sigma=h$,
\begin{align*}
B=\tilde{d}\uparrow^{d^2+d-2}\tilde{d}\uparrow \sigma\\
=(h(e-1)+1)\uparrow^{h^2(e-1)^2+h(e-1)-2}(h(e-1)+1)\uparrow h.   
\end{align*}
Then $f_1,\ldots,f_h,f$ lie in $S=K[y_1,\ldots,y_B]$ for a suitable choice of regular sequence $y_1,\ldots,y_B$. According to the above discussion, we can form the rings $S,T$ and the $T$-ideal $(f_1,\ldots,f_h,f)T$. We have
$$\textup{pd}(R/P)=\textup{pd}(S/Q) \leq B.$$
Now
$$\reg(T/QT) \leq \max\{\reg(T/(f_1,\ldots,f_h)T),\reg(T/(f_1,\ldots,f_h,f)T)\}+1.$$
Since $f_1,\ldots,f_h$ is a regular sequence,
$$\reg(T/(f_1,\ldots,f_h)T)=h(e-1).$$
Also by Lemma \ref{CS1},
$$\reg(T/(f_1,\ldots,f_h,f)T) \leq (2h(e-1))^{2^{B-2}}.$$
Since $(2h(e-1))^{2^{B-2}} \geq h(e-1) = \reg(T/(f_1,\ldots,f_h)T)$, $(2h(e-1))^{2^{B-2}}$ serves as an upper bound for the two terms inside the maximum expression. So
$$\reg(R/P)=\reg(T/QT) \leq (2h(e-1))^{2^{B-2}}+1.$$
Note that $h(e-1)+1<eh$ and $h^2(e-1)^2+h(e-1)<h^2e^2$, so
$$B \leq (eh)\uparrow^{e^2h^2-2}(eh)\uparrow h.$$
Also we have $2eh \leq (eh)^2$ and $\sigma \to eh\uparrow2\uparrow\sigma$ is strictly increasing. Therefore,
\begin{align*}
\reg(R/P) \leq (2eh)\uparrow2\uparrow((eh)\uparrow^{e^2h^2-2}(eh)\uparrow h-2)+1 \\
\leq (eh)\uparrow2\uparrow((eh)\uparrow^{e^2h^2-2}(eh)\uparrow h-1)+1\\
\leq (eh)\uparrow2\uparrow(eh)\uparrow^{e^2h^2-2}(eh)\uparrow h\\
\leq (eh)\uparrow^{e^2h^2-1}(h+3).
\end{align*}
Now we assume $K$ is algebraically closed. By going modulo a linear form in $P_1$, we reduce the height by 1 while preserving $e(R/P)$ and $\reg(R/P)$. Iterating, we may assume $P$ is nondegenerate, and then by Lemma \ref{7.3} we have $h \le e-1$. If $h \geq 2$, the above inequality leads to
\begin{align*}
\reg(R/P) \leq (e^2)\uparrow^{e^4-1}(e+2).
\end{align*}
If $h=0$ in this case, then $\reg(R/P)=0$. If $h=1$ then $P$ is a principal prime ideal generated by a form of degree $e=\deg(R/P)$, so $\reg(R/P)=e-1$. So when $h=0$ or $h=1$, the above bound still holds, and we are done.
\end{proof}

\bibliographystyle{plain}
\bibliography{referencetrans}{}

\end{document}